\documentclass[11pt]{article}

\usepackage{latexsym}
\usepackage{epsfig}
\usepackage{color}
\usepackage{amsmath}
\usepackage{amsthm}
\usepackage{mathrsfs}
\usepackage{array}
\usepackage{amssymb}
\usepackage{graphics}
\usepackage{epsfig}
\usepackage{subcaption}
\usepackage{frcursive}
\usepackage[utf8]{inputenc}
\usepackage[english]{babel}
\usepackage[normalem]{ulem}

\newcommand{\RNum}[1]{\uppercase\expandafter{\romannumeral #1\relax}}
\newcommand{\Rnum}[1]{\lowercase\expandafter{\romannumeral #1\relax}}

\newcommand{\be}{\begin{eqnarray}}
\newcommand{\ee}{\end{eqnarray}}
\newcommand{\beno}{\begin{eqnarray*}}
\newcommand{\eeno}{\end{eqnarray*}}

\newtheorem{theorem}{Theorem}[section]
\newtheorem{lemma}[theorem]{Lemma}
\newtheorem{proposition}[theorem]{Proposition}
\newtheorem{corollary}[theorem]{Corollary}

\newtheorem{assumption}[theorem]{Assumption}

\theoremstyle{definition}
\newtheorem{definition}[theorem]{Definition}

\theoremstyle{remark}
\newtheorem{remark}[theorem]{Remark}
\newtheorem{example}[theorem]{Example}

\newcommand{\rmi}{{\rm (i)$\>\>$}}
\newcommand{\rmii}{{\rm (ii)$\>\>$}}

\def \A{\mathbb{A}}

\def \E{\mathbb{E}}
\def \F{\mathbb{F}}

\def \P{\mathbb{P}}

\def \R{\mathbb{R}}
\def \S{\mathbb{S}}
\def \T{\mathbb{T}}

\def \Z{\mathbb{Z}}

\def\Ac{{\cal A}}

\def\Dc{{\cal D}}

\def\Fc{{\cal F}}

\def\Hc{{\cal H}}

\def\Lc{{\cal L}}

\def\Pc{{\cal P}}

\def\Uc{{\cal U}}

\def\Wc{{\cal W}}

\def\Pb{{\overline \P}}

\def\Ph{{\hat P}}

\def\Xh{{\widehat X}}

\def\x{\times}

\def\={\;=\;}
\def\.{\;.}

\def\eps{\varepsilon}

\def\1{{\bf 1}}

\renewenvironment{proof}{{\bfseries Proof.}}{\qed}

\def\Om{\Omega}
\def\om{\omega}

\def\Omb{\overline{\Om}}

\def\Fcb{\overline{\Fc}}

\def\0{\mathbf{0}}

\def\normeL2#1{\left\|{#1}\right\|_{L^2}}

\def\Fbb{\overline \F}

\def\Pcb{\overline \Pc}

\def\Omh{\widehat \Om}
\def\Pch{\widehat \Pc}
\def\Ph{\widehat \P}
\def\Fch{\widehat \Fc}

\begin{document}
\title{Duality and approximation of stochastic optimal control problems under expectation constraints
}

\author{
Laurent Pfeiffer\footnote{Inria  and  CMAP  (UMR  7641),  CNRS,  Ecole  Polytechnique,  Institut  Polytechnique  de  Paris.  Email: laurent.pfeiffer@inria.fr}
\and Xiaolu  Tan\footnote{Department of Mathematics, The Chinese University of Hong Kong. Email: xiaolu.tan@cuhk.edu.hk}
\and Yulong Zhou\footnote{School of Mathematics, Sun Yat-Sen University, Guangzhou, China. Email: zhouyulong@mail.sysu.edu.cn}
}

\date{\today}

\maketitle

\begin{abstract}
We consider a continuous time stochastic optimal control problem under both equality and inequality constraints on the expectation of some functionals of the controlled process.
Under a qualification condition,
we show that the problem is in duality with an optimization problem involving the Lagrange multiplier associated with the constraints.
Then by convex analysis techniques, we provide a general existence result and some a priori estimation of the dual optimizers.
We further provide a necessary and sufficient optimality condition for the initial constrained control problem.
The same results are also obtained for a discrete time constrained control problem.
Moreover, under additional regularity conditions, it is proved that the discrete time control problem converges to the continuous time problem, possibly with a convergence rate.
This convergence result can be used to obtain numerical algorithms to approximate the continuous time control problem,
which we illustrate by two simple numerical examples.
\end{abstract}



\vspace{2mm}


\section{Introduction}

We study in this article a stochastic optimal control problem under expectation constraints, in a path-dependent framework.
Concretely, let the controlled diffusion process be governed by the dynamics
$$
dX^{\alpha}_t = \mu(t, X^{\alpha}_{t \wedge \cdot}, \alpha_t) dt + \sigma(t, X^{\alpha}_{t\wedge \cdot}, \alpha_t) dB_t,
$$
where $B$ is a Brownian motion. Given functionals $\Phi$, $\Psi_0$, and $\Psi_1$, we consider the constrained optimization problem
\begin{equation} \label{eq:the_problem_to_solve}
\inf_{\alpha}~ \E \big[ \Phi(X^{\alpha}_{\cdot}) \big],
~~\mbox{subject to:}~~
\E[ \Psi_0(X^{\alpha}_{\cdot})] = 0
~\mbox{and}~
\E[\Psi_1(X^{\alpha}_{\cdot})] \le 0.
\end{equation}
In the above problem, $X^{\alpha}_{t \wedge \cdot}$ denotes the stopped process $(X^{\alpha}_{t \wedge s})_{s \in [0,T]}$
and the coefficient functions $\mu$, $\sigma$, $\Phi$, $\Psi_0$ and $\Psi_1$ may depend on the whole paths of $(X^{\alpha}_s)_{s \in [0,T]}$.

To illustrate the applications of \eqref{eq:the_problem_to_solve},
let us first mention the case of probability constraints. 
Such constraints can typically be written as follows:
\begin{equation} \label{eq:proba_constraint}
\mathbb{P} \big[ h(X_\cdot^\alpha) \leq 0 \big] \leq \gamma,
\end{equation}
where $h$ is vector-valued and $\gamma \in (0,1)$.
The constraint \eqref{eq:proba_constraint} can be put in the form $\E[\Psi_1(X^{\alpha}_{\cdot})] \le 0$ by setting
$
\Psi_1(X_\cdot^\alpha)
= \mathbf{1}_{h(X_\cdot^\alpha) \leq 0} - \gamma,
$
where $\mathbf{1}$ denotes the indicator function. Note that such a function $\Psi_1$ is in general discontinuous.
Probability constraints are typically used to take into account undesirable events, which may however arise with nonzero probability whatever the employed control process. We refer to Shapiro, Dentcheva and Ruszczy\'nski \cite[Chapter 4]{SDR} for a general presentation.
Another important application of problem \eqref{eq:the_problem_to_solve} may concern the characterization of the Pareto front in a multi-objective setting, see e.g.\@ Yong and Zhou \cite[Chapter 3]{YongZhou}.
More examples of applications of problem \eqref{eq:the_problem_to_solve}, such as the control problem under state constraints, the semi-static hedging problem in finance, will also be discussed in Section \ref{sec:Formulation}.

In the literature, different approaches have been proposed to study the constrained control problem \eqref{eq:the_problem_to_solve}.
The first one consists in introducing an additional process, denoted by $(Y^0_t, Y^1_t)_{t \in [0,T]}$, with initial value $Y^i_0 = 0$ and final value $Y^i_T = \Psi_i(X^{\alpha}_{\cdot})$, $i=0,1$, and having a martingale or submartingale dynamics.
Then one can either reformulate the initial problem as a stochastic target problem which can be studied by the geometric dynamic programming approach (see e.g.\@ Soner and Touzi \cite{SonerTouzi}, Bouchard, Elie and Imbert \cite{BEI},  Bouchard and Nutz \cite{BouchardNutz}, Chow, Yu and Zhou \cite{CYZ}, etc.),
or one can use the so-called \emph{level-set approach} to reformulate it into an optimization problem over a family of (unconstrained) singular control problems (see e.g.\@ Bokanowski, Picarelli and Zidani \cite{BPZ} and the references therein). 
In the second main approach, one looks for necessary optimality conditions in the form of Pontryagin's maximum principle, involving a Lagrange multiplier, see e.g.\@ Yong and Zhou \cite[Chapter 3]{YongZhou} and Bonnans and Silva \cite[Section 5]{BonnansSilva}. Such optimality conditions are obtained with the variational technique, which consists in calculating a first-order Taylor expansion of the cost with respect to small perturbations of the control. The constraints are usually tackled with Ekeland's principle or with the Lyusternik-Graves theorem (under a qualification condition).

In the present work, we follow a \emph{duality approach} by considering the following dual problem:
\begin{equation} \label{eq:the_dual_problem}
\sup_{\lambda_0 \in \R, \, \lambda_1 \ge 0} \, \inf_{\alpha} \ \ \E \big[ \Phi(X^{\alpha}_{\cdot}) + \lambda_0 \Psi_0(X^{\alpha}_{\cdot}) + \lambda_1 \Psi_1(X^{\alpha}_{\cdot})  \big].
\end{equation}
The above problem can be formally obtained by writing \eqref{eq:the_problem_to_solve} as a saddle-point problem and then by exchanging the $\inf$ and $\sup$ operators.
Our duality result, Theorem \ref{thm:duality}, states that there is no duality gap between \eqref{eq:the_problem_to_solve} and \eqref{eq:the_dual_problem}, i.e.\@ the two problems have the same value.
Such a duality result should be in line with Kantorovich's duality for the optimal transport (OT) problem,
recalling that the marginal constraint in OT is equivalent to infinitely many expectation constraints.
Let us refer in particular to Mikami and Thieullen \cite{MikamiThieullen}, Tan and Touzi \cite{TanTouzi} for an optimal transport problem along controlled diffusion processes.
In contrast to the duality theory in OT, which is generally based on the compactness of the control space satisfying the marginal constraints, together with the regularity of the reward/cost function,
we will develop and explore here the Lagrange relaxation approach which was utilized in Pfeiffer \cite{Pfeiffer} for a discrete time control problem.

In a first step, we develop the Lagrange relaxation approach in an abstract framework,
for a problem in the form:
$$
\inf_{\P \in \Pc} \E^{\P}[ \xi],
\quad \text{subject to:}\quad
\E^{\P}[\zeta_0] = 0,~\E^{\P}[\zeta_1] \le 0,
$$
where $\Pc$ is a family of probability measures on an abstract measurable space $(\Om, \Fc)$.
By assuming that $\Pc$ is convex, and using a qualification condition,
we obtain an abstract duality result.
To reduce the constrained stochastic control problem \eqref{eq:the_problem_to_solve} to the above abstract framework,
we use a weak formulation of the control problem by considering the law of controlled processes on the canonical space.
Such an approach allows to study the problem in a very general framework,
in both discrete time and continuous time setting, Markovian or path-dependent,
and requires no regularity on the reward/cost function.

In a second step, we explore properties of the initial constrained optimization problem as well as its dual problem.
First, we provide a general existence result on the optimizers of the dual problem,
together with some estimation.
Next, using the dual optimizer, we obtain a necessary and sufficient optimality condition for the  initial constrained problem.
To the best of our knowledge, such an optimality condition in the non-Markovian setting should be new in the literature.
Restricting to the Markovian setting, our optimality condition reduces to the \emph{optimality conditions in variational form}  in Pfeiffer \cite[Section 4]{pfeiffer19},
and could lead naturally to the necessary condition in the maximum principle in \cite{YongZhou,BonnansSilva} described previously.

Finally, we also consider a discrete time control problem under constraints
and study its convergence to the continuous time problem.
Using the (Wasserstein) weak convergence technique, together with some a priori estimation on the dual problem,
we are able to provide a very general approximation result.
Moreover, with the a priori bound of the dual optimizer, and based on the approximation results for control problem without constraints,
we can obtain some results on the convergence rate.
This approximation result can be used to suggest numerical algorithms to approximate the continuous time problem.
In particular, it can be considered as an extension of the classical weak convergence approach of Kushner and Dupuis \cite{KushnerDupuis} to approximate control problems without constraints.
Notice also that similar numerical algorithms have been used in \cite[Section 5]{TanTouzi} for a stochastic optimal transport problem.
Our technique, based on the (Wasserstein) weak convergence arguments and a priori dual optimizer estimation, leads to a much more general convergence result, 
and with possible convergence rate, which is new in the literature.

The dual approach investigated in this work offers two main advantages.
	Compared with the geometric dynamic programming approach \cite{BEI}, the level-set approach \cite{BPZ}, and the maximum principle approach \cite{BonnansSilva}, which need to stay in a Markovian context, it applies in a general non-Markovian context (for both discrete time and continous time settings). Moreover, it allows to obtain novel numerical approximation methods. In the Markovian case, we can prove some convergence rates; this is another novelty in comparison with the previous literature. See also Remark \ref{rem:discussion_numer_method} for more discussions.

The rest of the paper is organized as follows.
In Section \ref{sec:Formulation}, we provide a detailed weak formulation of the control problem under expectation constraints, in both continuous time version and discrete time version.
Then in Section \ref{sec:MainResults}, we give the technical conditions as well as the main results, including the main duality and approximation results.
In Section \ref{sec:Numerical}, we provide two simple numerical examples,
by considering a linear quadratic control problem under constraints. 
The technical proofs of the main results are completed in Section \ref{sec:Proofs}.

\section{Optimal control problems under expectation constraints}
\label{sec:Formulation}

We provide here a weak formulation of the optimal control problem under expectation constraints, in both continuous time version and discrete time version,
as well as the corresponding dual formulations.

\subsection{An optimal control problem under constraints in continuous time}

Let $n > 0$ be a positive integer, $T >0$ and $\Om := C([0,T], \R^n)$ denote the canonical space of all $\R^n$-valued continuous paths on $[0,T]$,
with canonical process $X$. 
Let $\| \cdot \|$ denote the uniform convergence norm on $\Om$.
Let $(A, \rho)$ be a nonempty Polish space and  $(\mu, \sigma) \colon [0,T] \x \Om \times A \to \mathbb{R}^n \times \mathbb{S}^n$ be the coefficient functions for the controlled diffusion processes,
where $\mathbb{S}^n$ denotes the collection of all $n \times n$-dimensional matrices.
We assume that $\mu$ and $\sigma$ are non-anticipative, 
in the sense that $(\mu, \sigma)(t, \om, a) = (\mu, \sigma)(t, \om_{t \wedge \cdot}, a)$ for all $(t, \om, a) \in [0,T] \x \Om \x A$.
Moreover, for  some fixed point $a_0 \in A$, there exists a constant $K > 0$ such that, for all $(t, \om, a) \in [0,T] \x \Om \x A$,
\begin{equation} \label{eq:growth_condition}
 \big| (\mu, \sigma)(t, \om, a) \big| \le K (1 + \|\om_{t \wedge \cdot}\| + \rho(a, a_0) \big).
\end{equation}
We also fix $x_0 \in \R^n$ as initial position of the controlled process.

\begin{definition} \label{def:weak_ctrl}
A term $\gamma = (\Om^{\gamma}, \Fc^{\gamma}, \F^{\gamma}=(\Fc^{\gamma}_t)_{t \in [0,T]}, \P^{\gamma}, X^{\gamma}, \alpha^{\gamma}, B^{\gamma})$ is a weak control term 
if $ (\Om^{\gamma}, \Fc^{\gamma}, \F^{\gamma}, \P^{\gamma})$ is a filtered probability space, 
in which $B^{\gamma}$ is an $n$-dimensional standard Brownian motion, 
$\alpha^{\gamma}$ is an $A$-valued predictable process satisfying
\be \label{eq:Moment}
\E^{\P^{\gamma}} \Big[ \|X^{\gamma} \|^2 +  \int_0^T \big( \rho(\alpha^{\gamma}_t, a_0) \big)^2 dt \Big] < \infty,
\ee
and $X^{\gamma}$ is an $\R^n$--valued adapted continuous process such that
\be \label{eq:SDE_strong}
\qquad X^{\gamma}_t = x_0 + \int_0^t \mu(s, X^{\gamma}_{s\wedge \cdot}, \alpha^{\gamma}_s) ds + \int_0^t  \sigma(s, X^{\gamma}_{s\wedge \cdot}, \alpha^{\gamma}_s) d B^{\gamma}_s, ~~t \in [0,T],~~ \P^{\gamma}\text{-a.s.}
\ee
Let us denote by $\Gamma$ the collection of all weak control terms.
\end{definition}

\begin{remark}
$\mathrm{(i)}$ The above control is called a weak control because the probability space and the Brownian motion are not a priori fixed, and more importantly, 
the control process may not be adapted to the filtration generated by the Brownian motion.
Using the martingale problem, one can interpret the weak controls as probability measures on the canonical space,
and the set of weak control (measures) is convex (but may not be closed).
Such a convexity property will be essential in the proof of the duality result.
The weak control may be different from the classical strong control, where a fixed probability space and a Brownian motion are given.
See also Section \ref{subsec:main_result_discuss} for a detailed comparaison between the strong and weak formulation of the constrained control problem. \\[1mm]
$\mathrm{(ii)}$ The growth condition \eqref{eq:growth_condition}, together with the integrability condition \eqref{eq:Moment}, ensures that  the (stochastic) integrals in \eqref{eq:SDE_strong} are well defined. 
In particular, one has
$$
\int_0^T \big( \big| \mu(s, X^{\gamma}_{s\wedge \cdot}, \alpha^{\gamma}_s) \big| + \big|  \sigma(s, X^{\gamma}_{s\wedge \cdot}, \alpha^{\gamma}_s) \big|^2 \big) ds < \infty, ~\P^{\gamma} \mbox{--a.s.}
$$
At the same time, there is a freedom to choose the metric $\rho$ for the Polish space $A$.
In particular, when $(\mu, \sigma)$ are uniformly bounded, one can choose a uniformly bounded metric $\rho$ so that \eqref{eq:Moment} holds true for any process $\alpha^{\gamma}$.
\end{remark}

Let $m, \ell \ge 0$ and let $\Phi \colon \Om \to \R$ and $\Psi_i \colon \Om \to \R$, $i=1, \cdots, m+\ell$
be functionals defined on $\Om$, 
we consider the following optimization problem under expectation constraints:
\be \label{eq:def_V} 
V(z)
:= 
\inf_{\gamma \in \Gamma(z)} \E^{\P^{\gamma}} \big[ \Phi \big( X^{\gamma}_{\cdot} \big) \big],
\ee
where for all $z \in \R^m \x \R^{\ell}$,
\be \label{eq:def_Uc}
\Gamma(z) 
&:=& 
\big\{\gamma \in \Gamma ~: 
\E^{\P^{\gamma}} \big[ \Psi_i \big( X^{\gamma}_{\cdot} \big) \big] = z_i, ~i=1, \cdots, m,  \nonumber \\
&&
~~~~~~~~~~~~~\mbox{and}~~
\E^{\P^{\gamma}} \big[\Psi_{m+j} \big(X^{\gamma}_{\cdot} \big) \big] \le z_{m+j}, 
~ j=1,\cdots, \ell 
\big\}.
\ee
We are actually mostly interested in the case $z=0$ but directly consider a parameterized version of the problem in \eqref{eq:def_V}, for conviniency. Integrability condtions on $\Phi$ and $\Psi$ will be specified in Section \ref{subsec:assumptions}.

\begin{remark}
We can consider more general costs and constraints, such as
$$
\inf_{\gamma \in \Gamma(z)} \E^{\P^{\gamma}}\Big[
\int_0^T L\big(s, X^{\gamma}_{s \wedge \cdot}, \alpha^{\gamma}_s \big) ds + \Phi \big( X^{\gamma}_{\cdot} \big)
\Big],
$$
for some function $L \colon [0,T] \x \Om \x A \to \R$.
By introducing an $n +1$ dimensional process $\widetilde X$ defined by $\widetilde X^i := X^i$ for $i=1, \cdots, n$ and $\widetilde X^{n+1}_t := \int_0^t L(s, X^{\gamma}_{s \wedge \cdot}, \alpha^{\gamma}_s) ds$
and the corresponding new criteria function, one can reduce the problem to the formulation \eqref{eq:def_V}.
From a numerical approximation point of view, it is always better to stay in the minimal dimension context.
Here we would like to use this formulation (\`a la Mayer) to study the problem for ease of presentation
and leave the dimensional reduction to the numerical implementation step.
\end{remark}
 
\paragraph{A dual formulation}

By penalizing the constraints, it is clear that for $z =0$,
$$
V(0) 
= 
\inf_{\gamma \in \Gamma} \sup_{\lambda \in \R^m \x \R^\ell_+}  \E^{\P^{\gamma}} \big[ \Upsilon(X^{\gamma}_{\cdot},\lambda) \big],
$$
where
\begin{equation*}
\Upsilon(\omega,\lambda)= \Phi(\omega) + \lambda \cdot \Psi(\omega).
\end{equation*}
It follows that a dual formulation of the optimization problem $V(0)$ can be given by
\be \label{eq:def_D}
D(0) := \sup_{\lambda \in \R^m \x \R^\ell_+} d(\lambda),
~~~\mbox{where}~~
d(\lambda) := \inf_{\gamma \in \Gamma} \ \E^{\P^{\gamma}} \big[ \Upsilon(X^{\gamma}_{\cdot},\lambda) \big].
\ee	

\paragraph{Some examples}
We also provide some examples of applications of the above constrained optimization problem.

\begin{example}[Optimal control under state constraint] \label{example:state_constraints}
The optimal control problem under state constraint has been studied by different papers (see e.g.\@ \cite{BPZ} and the references therein), it can be formulated as follows:
\be \label{eq:Ctrl_State_Constraint}
\inf \big\{ \E\big[ \Phi(X^{\gamma}_T) \big] ~: \gamma  \in \Gamma, ~ X^{\gamma}_t \in E ~\mbox{a.s.\@ for all}~t \in [0,T]
\big\},
\ee
where $E$ is a closed convex subset of  $\R^n$.
Let us denote by $d(x,E)$ the distance between $x$ and $E$, and let us define
$$
\Psi(X^{\gamma}_{\cdot}) ~:=~ \max_{0 \le t \le T} d (X^{\gamma}_t , E),
$$
then Problem \eqref{eq:Ctrl_State_Constraint} is equivalent to the following problem under expectation constraint (in a path-dependent version):
$
\inf \big\{ 
\E\big[ \Phi(X^{\gamma}_T) \big] 
~: \gamma \in \Gamma, ~ \E\big[ \Psi(X^{\gamma}_{\cdot}) \big] = 0
\big\}.
$
\end{example}

\begin{example}[Semi-static super-replication problem in finance] \label{exam:vol_uncertain}
Let us consider an uncertain volatility financial model, where the interest rate $r=0$ and a risky asset price follows the dynamics
$dS^{\gamma}_t = \sigma^{\gamma}_t S^{\gamma}_t dW_t$ with an adapted process  $\sigma^{\gamma} = (\sigma^{\gamma}_t)_{0 \le t \le T}$.
Denote by $\Gamma$ the collection of all models $\gamma = (S^{\gamma}, \sigma^{\gamma})$ such that $\sigma^{\gamma}$ takes value in $[\underline \sigma, \overline \sigma]$, 
and assume that the market model is uncertain (unknown), but the class of all possible market dynamics is given by
$\{ (S^{\gamma}_t)_{t \in [0,T]} ~: \gamma  \in \Gamma \}$.

A dynamic trading strategy is an adapted process $H = (H_t)_{0 \le t \le T}$ and $(H \cdot S)_T := \int_0^T H_t dS^{\gamma}_t$ is the Profit \& Loss of the dynamic strategy. Denote by $\Hc$ the class of all dynamic strategies.
On the market, there are two vanilla options of payoff $\Psi_1(S^{\gamma}_T)$ and $\Psi_2(S^{\gamma}_T)$ at maturity $T$.
Assume that an agent is allowed to buy or to sell the first option with price $z_1$, and is allowed to buy the second option with price $z_2$ (but cannot short/sell it).
There is another option with payoff $\Phi(S^{\gamma}_{\cdot})$, and the agent aims to super-replicate it with both dynamic trading strategy $H$ and static strategy on options $\Psi_1$ and $\Psi_2$.
Then the collection of all super-replication strategies is given by
\begin{align*}
\Dc := & 
\big \{(x, H, \lambda_1, \lambda_2) \in \R \x \Hc \x \R \x \R_+  ~:\\
& \quad
x+ (H \cdot S^{\gamma})_T + \lambda_1 (\Psi_1(S^{\gamma}_T) - z_1) + \lambda_2 (\Psi_2(S^{\gamma}_T) - z_2) \ge \Phi(S^{\gamma}_{\cdot}), \mbox{a.s.\@ } \forall \gamma  \in \Gamma
\big\},
\end{align*}
and the minimal super-replication cost of option $\Phi$ is given by
$$
\inf \big\{ x ~: (x,H, \lambda_1, \lambda_2) \in \Dc
\big\}.
$$

Next, by the duality theory (see e.g.\@ Denis and Martini \cite{DenisMartini}) in finance, the above problem has a dual formulation:
$$
\inf_{(\lambda_1, \lambda_2) \in \R \x \R_+} \ \,
\sup_{\gamma  \in \Gamma} \ \,
\E \big[ \Phi(S^{\gamma}_{\cdot}) - \lambda_1(\Psi_1(S^{\gamma}_T) - z_1) - \lambda_2(\Psi_2(S^{\gamma}_T) -z_2) \big],
$$
which is equivalent to a constrained optimal control problem:
$$
\sup_{\gamma  \in \Gamma} \ \big\{ \E [\Phi(S^{\gamma}_{\cdot})] ~: \gamma  \in \Gamma, ~ \E [\Psi_1(S^{\gamma}_T)] = z_1,~\E [\Psi_2(S^{\gamma}_T)] \le z_2
\big\}.
$$
Notice that the above application in finance has also motivated the so-called optimal Skorokhod embedding problem, and the martingale optimal transport problem which has recently generated an important literature (see e.g.\@ \cite{BCH} among many others). 
\end{example}

\subsection{A discrete time version of the constrained control problem}

Let us introduce a discrete time version of the constrained control problem, with the objective to approximate the continuous time problem \eqref{eq:def_V}.

Let $N > 0$ be an integer and $h:= \frac{T}{N}$, denote $t_k := k h$ and $\T_h := \{t_0, t_1, \cdots, t_N\}$.
Let $H_h: \T_h \x \Om \x A \x [0,1] \to \R^n$ be a function satisfying
$$
H_h(t_k, \om, a, u) = H_h(t_k, \om_{t_k \wedge \cdot}, a ,u),
~~\mbox{for all}~
(t_k, \om, a , u) \in \T_h \x \Om \x A \x [0,1].
$$

\begin{definition} \label{def:Control_h}
We say that
$$
\gamma 
~=~
\big( \Om^{\gamma}, \Fc^{\gamma}, \F^{\gamma}, \P^{\gamma}, \alpha^{\gamma}, X^{\gamma}, U^{\gamma}
\big)
$$
is a weak discrete time control term
if $(\Om^{\gamma}, \Fc^{\gamma}, \F^{\gamma}, \P^{\gamma})$ is a filtered probability space,
equipped with an $A$-valued process $\alpha^{\gamma}$, an $\R^n$--valued process $X^{\gamma}$, and a $[0,1]$--valued process $U^{\gamma}$, which are all adapted,
and for each $k=1, \cdots, N$, $U_{t_k}^{\gamma}$ is of uniform distribution on $[0,1]$ and independent of $\Fc^{\gamma}_{t_{k-1}}$,
and finally, $X^{\gamma}_0 = x_0$,
\begin{equation} \label{eq:X_gamma_dynamic}
X^{\gamma}_{t_k +1} 
~=~
X^{\gamma}_{t_k} 
~+~ 
H_h \big(t_k, \widehat X^{\gamma}_{t_k \wedge \cdot}, \alpha^{\gamma}_{t_{k}}, U^{\gamma}_{t_{k+1}} \big),
\end{equation}
where $\widehat X^{\gamma}$ denotes the continuous time process on $[0,T]$ obtained by linear interpolation of $(X^{\gamma}_{t_i})_{t_i \in \T_h}$.
\end{definition}

Let us denote by $\Gamma_h$ the collection of all weak discrete time control rules,
and by $\Gamma_h(z) \subset \Gamma_h$, for $z \in \R^{m+\ell}$, the subset of all controls $\gamma$ such that
\begin{equation} \label{eq:def_Uc_h}
	\E^{\P^{\gamma}}\big[ \Psi_i \big( \widehat X^{\gamma}_{\cdot}\big) \big] = z_i,
	~i = 1, \cdots, m,
	~~~\mbox{and}~~
	\E^{\P^{\gamma}}\big[ \Psi_{m+j} \big( \widehat X^{\gamma}_{\cdot}\big) \big] \le z_{m+j},
	~j=1, \cdots, \ell.
\end{equation}
Then a weak formulation of the discrete time constrained problem is given by
\be \label{eq:def_Vh}
V_h(z) ~:=  \inf_{\gamma \in \Gamma_h(z)}  \E^{\P^{\gamma}} \big[ \Phi \big(\widehat X^{\gamma}_{\cdot} \big) \big].
\ee

\paragraph{A dual formulation}

As in \eqref{eq:def_D}, we introduce the corresponding dual problem for the discrete time optimization problem, given by
\be \label{eq:dual_h}
D_h(0) :=  \sup_{\lambda \in \R^m \x \R^{\ell}_+}  d_h(\lambda),
~~~\mbox{where}~~
d_h(\lambda) :=   \inf_{\gamma \in \Gamma_h}  \E^{\P^{\gamma}} \big[ \Upsilon(\widehat X^{\gamma}_{\cdot},\lambda \big) \big].
\ee

\begin{remark}
In the definition of the weak control term in Definition \ref{def:Control_h}, if we fix the filtration to be the one generated by $(U^{\gamma}_k)_{k=1, \cdots, N}$, i.e.\@ $\Fc^{\gamma}_k = \sigma( U^{\gamma}_i, ~i \le k)$,
then it can be considered as a strong version of the control rule.
For the problem $d_h(\lambda)$ in \eqref{eq:dual_h}, it is equivalent to consider only strong control rules by the dynamic programming principle.
However, it may change the value of control problem under constraints in \eqref{eq:def_Vh}.
\end{remark}
	
\section{Main results}
\label{sec:MainResults}

We now provide the  assumptions and main results of the paper, including our duality and approximation results.
We then also discuss the numerical resolution of the discrete time (dual) constrained control problem.

\subsection{Assumptions} \label{subsec:assumptions}

The first assumption is on the integrability of the  functionals $\Phi$ and $\Psi$,
and a Robinson qualification condition on the constraints in both continuous and discrete time setting.

\begin{assumption} \label{assum:qualification}
$\mathrm{(i)}$
The random variables $\Phi(X^{\gamma}_{\cdot})$, $\Psi_i(X^{\gamma}_{\cdot})$,  $\Phi(\Xh^{\gamma_h}_{\cdot})$ and $\Psi_i(\Xh^{\gamma_h}_{\cdot})$ are integrable,
for all $i=1, \cdots, m+\ell$, $\gamma \in \Gamma$ and $\gamma_h \in \Gamma_h$.
Moreover, there is a constant $M > 0$ such that
$$
\sup_{\gamma \in \Gamma} ~\Big| \E^{\P^{\gamma}}[\Phi(X^{\gamma}_{\cdot})]  \Big| ~\le~ M
~~\mbox{and}~~
\sup_{\gamma_h \in \Gamma_h}~ \Big| \E^{\P^{\gamma_h}}[\Phi(\Xh^{\gamma_h}_{\cdot})]  \Big| ~\le~ M.
$$		

\noindent $\mathrm{(ii)}$ There is a constant $\eps > 0$ such that
\be \label{eq:qualification_conti_time}
B_{\eps}(0) \subset \mathrm{Conv}\big\{z \in \R^{m+ \ell} ~: \exists \gamma \in \Gamma, ~\mbox{s.t.}~ \E^{\P^{\gamma}}[\Psi(X^{\gamma}_{\cdot})] + z \in \{0_m\} \x \R^\ell_{-} \big\},
\ee
and
\be \label{eq:qualification_discrete_time}
\qquad B_{\eps}(0) \subset \mathrm{Conv}\big\{z \in \R^{m+ \ell} ~: \exists \gamma_h \in \Gamma_h, ~\mbox{s.t.}~ \E^{\P^{\gamma_h}}[\Psi(\widehat X^{\gamma_h}_{\cdot})] + z \in \{0_m\} \x \R^\ell_{-} \big\},
\ee
where $0_m \in \R^m$ denotes the original point in $\R^m$ and $B_{\eps}(0)$ denotes the closed ball in $\R^{m + \ell}$ with center 0 and radius $\eps$ for the supremum norm.
\end{assumption}

\begin{remark}
	\noindent $\mathrm{(i)}$
	In our technical proof, the boundedness condition in Assumption \ref{assum:qualification}$\mathrm{(i)}$ is only needed to ensure that
	\begin{equation} \label{eq:condit_bounded}
		| V(z) | \le M
		~~\mbox{and}~~
		|V_h(z)| \le M,
		~~\mbox{for all}~~
		z \in B_{\eps}(0).
	\end{equation} 
	In this abstract framework, $V(z)$ and $V_h(z)$ denote the value function of our main problems, which are unknown a priori.
	We therefore prefer to formulate a sufficient condition on $\Gamma$, $\Gamma_h$ and $\Phi$ as in Assumption \ref{assum:qualification}$\mathrm{(i)}$ to ensure \eqref{eq:condit_bounded}.
	In more concrete examples, it is certainly possible to work with weaker conditions that ensure  \eqref{eq:condit_bounded}, so that the results in the following still hold.

	\noindent $\mathrm{(ii)}$ The conditions in Assumption \ref{assum:qualification}$\mathrm{(ii)}$ are Robinson qualification conditions for the constrained optimization problem $V(z)$ in \eqref{eq:def_V} and $V_h(z)$ in \eqref{eq:def_Vh} with $z = 0$.
	They ensure that the equality and inequality constraints in \eqref{eq:def_Uc} and \eqref{eq:def_Uc_h} can be satisfied by some control $\gamma \in \Gamma$ or $\gamma_h \in \Gamma_h$ for $z = 0$.
	More importantly, with some flexibility, they are satisfied for all $z \in B_{\eps}(0)$.
	In Section \ref{subsec:main_result_discuss}, we will also provide a discussion on how to check this qualification condition.
\end{remark}

	The next assumption will be essentially used to deduce the convergence of the discrete time control problem to the continuous time problem as the time step $\Delta t \longrightarrow 0$.

\begin{assumption} \label{assum:approximation}
$\mathrm{(i)}$ The Polish space $A$ is compact, and the associated metric $\rho$ is uniformly bounded.
The function $\mu$ and $\sigma$ are continuous in all arguments and for some constant $K > 0$,
\begin{equation} \label{eq:Lip}
\|(\mu, \sigma)(t, \om, a) - (\mu, \sigma)(t, \om', a)\| \le K \|\om_{t \wedge \cdot} - \om'_{t \wedge \cdot}\|,
~~\mbox{for all}~
(t, \om, \om', a).
\end{equation}
Moreover,  the functions $\Phi(\om)$ and $\Psi_i(\om)$, $i=1, \cdots, m+\ell$ are continuous and bounded by $K(1+\|\om\|^2)$ for some constant $K > 0$.

	\vspace{0.5em}

	\noindent $\mathrm{(ii)}$  For every $(t_k, \om, a) \in \T_h \x \Om \x A$, with a random variable $U$ of uniform distribution $\Uc[0,1]$, one has
\begin{equation} \label{eq:consistency}
\E [H_h(t_k, \om, a, U) ] = \mu(t_k, \om, a) h,
~~~
\mathrm{Var} [H_h(t_k, \om, a, U) ] = \sigma \sigma^{\top}(t_k, \om, a) h.
\end{equation}
Moreover, there is some constant $C > 0$ such that $\E \big[ \big| H_h(t_k, \om, a, U) \big|^3 \big] \le C h^{3/2}$ for all $(t_k, \om, a) \in \T_h \x \Om \x A$ and $h > 0$.
\end{assumption}

\begin{remark}
	$\mathrm{(i)}$ The Lipschitz condition \eqref{eq:Lip} is the standard condition to ensure that the SDE \eqref{eq:SDE_strong} has a unique strong solution with a given control process $\alpha^{\gamma}$.

	\vspace{0.5em}
	
	\noindent $\mathrm{(ii)}$ The conditions in \eqref{eq:consistency} ensure that the discrete time controlled processes are good approximations of the continuous time controlled processes when the time step $h = T/N$ is small.
	To see this, let us consider the simple one-dimensional ($d=1$) example where $\mu \equiv \mu_0$ and $\sigma \equiv \sigma_0$ for some constants $\mu_0, \sigma_0 \in \R$.
	In this case, one has $X^{\gamma}_t = x_0 + \mu_0 t + \sigma_0 B^{\gamma}_t$,
	so that 
	$$
		\E [ X^{\gamma}_{t+h} - X^{\gamma}_t ] = \mu_0 h,
		~~\mbox{and}~~
		\mathrm{Var} [X^{\gamma}_{t+h} - X^{\gamma}_t] = \sigma_0^2 h.
	$$
	The condition $\E \big[ \big| H_h(t_k, \om, a, U) \big|^3 \big] \le C h^{3/2}$ ensures in addition that the set of distributions induced by the discrete time controls is relatively compact under the Wasserstein distance (which will be used in Section \ref{subsec:proof_cvg}).
\end{remark}

\subsection{Duality and approximation of the constrained control problems}

Our first main result is on the duality of  the two constrained problems,
together with the existence of dual optimizers and some estimation.
Let us denote
$$
\| \lambda \|_1 := \sum_{i=1}^{m+\ell} |\lambda_i|,
~~\mbox{for}~
\lambda = (\lambda_1, \cdots, \lambda_{m+\ell}) \in \R^{m + \ell}.
$$

\begin{theorem} \label{thm:duality}
Let Assumption \ref{assum:qualification} hold true.

	\noindent \rmi We have
$$
V (0) = D (0)
\quad \mbox{and} \quad
V_h (0) = D_h (0).
$$
Moreover, for the dual problem $D$ (resp.\@ $D_h$), there exist optimal solutions;
and any optimal solution $\lambda^* \in \R^m \x \R^{\ell}_+$ (resp.\@ $\lambda_h^* \in \R^m \x \R^{\ell}_+$) satisfies
$\| \lambda^* \|_{1} \le \frac{2M}{\eps}$ (resp.\@ $\| \lambda_h^* \|_{1} \le \frac{2M}{\eps}$). \\[1mm]
\noindent \rmii
Let $\gamma^* \in \Gamma(0)$. Then $\gamma^*$ is a (global) solution to the constrained problem \eqref{eq:def_V} if and only if there exists $\lambda^* \in \R^m \times \R^{\ell}_+$ such that the following conditions hold true:
\begin{align}
& \text{[Complementarity]} \hspace{-5mm} && \E^{\P^{\gamma^*}} \Big[ \sum_{i=1}^{m+\ell} \lambda^*_i \Psi_i(X^{\gamma^*}_{\cdot}) \Big] = 0, \label{eq:complementarity_true} \\[1em]
& \text{[Stationarity]} && \E^{\P^{\gamma^*}}\big[ \Upsilon(X^{\gamma^*}_{\cdot},\lambda^*) \big] 
=
\inf_{\gamma \in \Gamma} \E^{\P^{\gamma}}\big[ \Upsilon(X^{\gamma}_{\cdot},\lambda^*) \big].  \label{eq:min_pb_lagrangien}
\end{align}
In this case, $\lambda^*$ is a solution to the dual problem (i.e.\@ $d(\lambda^*)= D(0)$). \\[1mm]
\noindent {\rm (ii')$\>\>$}		
Let $\gamma_h^* \in \Gamma_h(0)$. Then $\gamma_h^*$ is a (global) solution to the constrained problem \eqref{eq:def_Vh} if and only if there exists $\lambda_h^* \in \R^m \times \R^{\ell}_+$ such that the following conditions hold true:
\begin{align}
& \text{[Complementarity]} \hspace{-0mm} && \E^{\P^{\gamma_h^*}} \Big[ \sum_{i=1}^{m+\ell} \lambda_{h,i}^* \Psi_i(\Xh^{\gamma_h^*}_{\cdot}) \Big] = 0, \label{eq:complementarity_true_disc} \\[1em]
& \text{[Stationarity]} && \E^{\P^{\gamma_h^*}}\big[ \Upsilon(\Xh^{\gamma_h^*}_{\cdot}, \lambda_h^*) \big] 
=
\inf_{\gamma_h \in \Gamma_h} \E^{\P^{\gamma_h}}\big[ \Upsilon(\Xh^{\gamma_h}_{\cdot}, \lambda_h^* ) \big]. \label{eq:min_pb_lagrangien_disc}
\end{align}
In this case, $\lambda_h^*$ is a solution to the discrete dual problem (i.e.\@ $d_h(\lambda_h^*)= D_h(0)$).	
\end{theorem}

\begin{remark}
The estimation of the dual optimizer $\lambda^*$ implies that it is enough to optimize $\lambda$ in a compact set 
$ B_{2M/\eps}(0) := \{ \lambda \in \R^m \x \R^{\ell}_+ ~: \|\lambda\|_1 \le 2M/\eps \}$ for the dual problem \eqref{eq:def_D}.
\end{remark}

\begin{remark}
Let $\gamma^*$ be an optimal solution to the initial constrained control problem, then it is also a solution to the unconstrained control problem $d(\lambda^*)$ with an optimal dual optimizer $\lambda^*$.
In the Markovian context, where $\mu$, $\sigma$, $\Phi$ and $\Psi_i$ depend only on the running process $X_t$ in place of the path $X_{t \wedge \cdot}$, and under some regularity condition,
one may deduce a necessary condition on $\gamma^*$ from the stochastic Pontryagin maximum principle of Peng \cite{Peng}.
This is exactly the extended stochastic Pontryagin maximum principle for the constrained optimal control problem, see e.g.\@ Yong and Zhou \cite{YongZhou}, Bonnans and Silva \cite{BonnansSilva}, etc.				
The optimality conditions provided in Theorem \ref{thm:duality} are in line with those of Pfeiffer in \cite{pfeiffer19} in a non-linear Markovian setting.
\end{remark}

\begin{remark}
	In the discrete time framework, 
	 consider the special case where $\E[H_h(t_k, \om, a, U)] = 0$ for all $a \in A$,
	so that the process $(X^{\gamma}_{t_k})_{t_i \in \T_h}$ in \eqref{eq:X_gamma_dynamic} is a martingale.
	When the family $\{ H_h(t_k, \om, a, U) ~: a \in A \}$ is rich enough,
	one can in addition dualize the martingale optimization problem $d_h(\lambda)$ in \eqref{eq:dual_h} as the biggest convex function dominated by $\Upsilon(\cdot, \lambda)$,
	so that one can obtain a dual formulation as a single maximization problem.
	In the discrete time context, the duality as well as it computation and applications in finance has been investigated in Miller \cite{miller}.
	
	In our main result, we do not look forward to dualize the problem $d_h(\lambda)$ in \eqref{eq:dual_h} since we stay in a context without necessarily a martingale structure for processes in $\Gamma_h$, as $H_h$ could be a very general kernel function.
	The computation of $d_h(\lambda)$ will also be generally more complicated than computing the convex envelop as in \cite{miller}.
\end{remark}

We next provide an approximation result under Assumption \ref{assum:approximation}, that is,
\begin{equation*}
V_h(0) = D_h(0)  \longrightarrow V(0) = D(0).
\end{equation*}
As the dual problems are equivalent to the optimization of $d(\lambda)$ or $d_h(\lambda)$ on a compact set $B_{2M/\eps}(0)$, one can expect to have a uniform convergence rate between $d_h(\lambda)$ and $d(\lambda)$ on $B_{2M/\eps}(0)$,
which implies a convergence rate between $D_h(0)$ and $D(0)$.
We formulate the convergence rate condition as follows.

\begin{assumption} \label{assum:cvg_rate}
There are constants $C > 0$ and $\rho > 0$ such that
$$
| d(\lambda) - d_h(\lambda) | \le C h^{\rho},
~~~\mbox{for all}~~\lambda \in B_{2M/\eps}(0) \subset \R^m \x \R^{\ell}_+.
$$	
\end{assumption}

\begin{remark} \label{rem:cvg_rate}
$\mathrm{(i)}$ Consider the Markovian context,
where
\begin{equation*}
(\mu, \sigma)(t,\om,a) = (\mu_0, \sigma_0)(t, \om_t, a)
\end{equation*}
for some $(\mu_0, \sigma_0) \colon [0,T] \x \R^n \x A \to \R^n \x \S^d$
and $\Phi(\om) = \Phi_0(\om_T)$, $\Psi_i(\om) = \Psi_{0,i}(\om_T)$ for some functions $\Phi_0 \colon \R^n \to \R$ and $\Psi_{0,i} \colon \R^n \to \R$.
Assume that $\mu_0, \sigma_0, \Phi_0, \Psi_{0,i}$ are all bounded continuous, and
$(\mu_0, \sigma_0)(t,x,a)$ is Lipschitz in $x$, $\frac12$-H\"older in $t$ uniformly on $[0,T] \x \R^n \x A$,
and $\Phi_0$, $\Psi_{0,i}$ are all Lipschitz in $x$,
and further that for any $\delta > 0$, there exists a finite subset $A_{\delta} \subset A$ such that
$$
\sup_{(t,x,a) \in [0,T] \x \R^n \x A}
\inf_{a' \in A_{\delta}} 
\Big( \big\| (\mu_0, \sigma_0) (t,x,a) - (\mu_0, \sigma_0)  (t,x,a') \big\| \Big)
~\le~
\delta.
$$
Then under Assumption \ref{assum:approximation}, the condition in Assumption \ref{assum:cvg_rate} holds true for some constant $C > 0$ (independent of $\lambda \in B_{2M/\eps}(0)$) and $\rho = \frac{1}{10}$.
The above convergence rate result was first proved by Krylov's \cite{Krylov} based on a shaking coefficient argument, and then improved by Barles and Jakobsen \cite{BarlesJakobsen}. \\[1mm]
\noindent $\mathrm{(ii)}$ In the non-Markovian context, assume that the sets
$\{\mu(t, \om, a) ~: a \in A\}$ and $\{\sigma \sigma^{\top}(t, \om, a) ~: a \in A \}$ are uniformly bounded and independent of $\om \in \Om$, and $\Phi$ and $\Psi_i$ are Lipschitz,
then under Assumption \ref{assum:approximation}, the condition in Assumption \ref{assum:cvg_rate} holds true for some constant $C > 0$ and $\rho = \frac{1}{8}$.
The above convergence rate result is obtained by Dolinsky \cite{Dolinsky} (see also Tan \cite{Tan}) using a strong invariance principle argument. \\[1mm]
\noindent $\mathrm{(iii)}$
In the non-Markovian context, assume that the volatility function $\sigma$ is not controlled so that the optimal control problem in $d(\lambda)$ can be reformulated as a BSDE (backward stochastic differential equation). Convergence results for numerical schemes of BSDE lead to a rate $\rho = \frac12$, see e.g.\@ Bouchard and Touzi \cite{BouchardTouzi}, Zhang \cite{Zhang}, etc.	
\end{remark}

\begin{theorem} \label{thm:cvg}
Let Assumptions \ref{assum:qualification} and \ref{assum:approximation} hold true.
Then
$$
D_h(0) = V_h(0) \longrightarrow V(0) = D(0),
~~~\mbox{as}~~
h \longrightarrow 0.
$$
Suppose in addition that Assumption \ref{assum:cvg_rate} holds true with some constants $C > 0$ and $\rho > 0$,
then
\be \label{eq:cvg_rate}
|V_h(0) - V(0)| = |D_h(0) - D(0)| ~\le~ C h^{\rho}.
\ee
\end{theorem}

\begin{remark}\label{rem:discussion_numer_method}
	As we will show in Section \ref{subsection:numerics}, the discrete time problem $V_h(0)$ (or $D_h(0)$) can serve as a numerical scheme for the continuous time problem $V(0)$.
	Although it is very natural to approximate a continuous time control problem under constraints by discretization methods, there are few results in the literature dealing with the convergence of such methods.

	
	For the optimal control problem under state constraints as formulated in Example \ref{example:state_constraints},
	Bokanowski, Picarelli and Zidani \cite{BPZ} use the so-called \emph{level-set approach} to reformulate it into an optimization problem over a family of (unconstrained) singular control problems,
	and then to characterize the value function of the singular control problem as unique viscosity solution to a HJB equation.
	Then using Barles and Souganidis' monotone scheme technique \cite{BS}, they prove the convergence of their numerical scheme to the solution of the HJB equation.

	In the context of an optimal transport problem (or equivalently an optimal control problem under distribution constraints), Tan and Touzi \cite{TanTouzi} provide a numerical algorithm based on the dual formulation together with a finite difference scheme.
	 After a truncation on the dual Lagrangian multiplier, they are able to obtain a convergence result (without convergence rate) of the discrete time algorithm.
	 The first general convergence result in Theorem \ref{thm:cvg} is based on the weak convergence technique, which allows to avoid the truncation of the dual multiplier, and technical convex conditions on the cost function as in \cite{TanTouzi}.

In \cite{pfeiffer19}, some more general non-linear constraints are considered, formulated on the probability distribution of the final state. In this context, there is no duality result which can be exploited numerically. Yet penalty approaches (e.g.\@ augmented Lagrangian algorithms) can be employed to find an optimal control satisfying certain optimality conditions. In the general context of \cite{pfeiffer19}, the convergence to a global solution is an open issue and there is no available discretization result, to our knowledge.
\end{remark}

\subsection{Further discussions} 
\label{subsec:main_result_discuss}

\paragraph{A strong formulation of the optimal control problem under constraints}
By fixing the probability space and the Brownian motion in the weak control term $\gamma$, one can obtain a strong formulation of the control problem.
This can be achieved by just fixing the filtration $\F^{\gamma}$ as the Brownian filtration, as all other processes are assumed to be adapted to it.

\begin{definition} \label{def:strong_ctrl}
$\mathrm{(i)}$ A weak control term 
\begin{equation*}		
\gamma = (\Om^{\gamma}, \Fc^{\gamma}, \F^{\gamma}=(\Fc^{\gamma}_t)_{t \in [0,T]}, \P^{\gamma}, X^{\gamma}, \alpha^{\gamma}, B^{\gamma})
\end{equation*}		
is called a strong control term if $\F^{\gamma}$ is the (augmented) filtration generated by $B^{\gamma}$. \\[1mm]
\noindent $\mathrm{(ii)}$ A strong control term $\gamma$ is called piecewise constant if  $\alpha^{\gamma}$ is piecewise constant, i.e.\@ $\alpha_{s} = \alpha_{t_i}$ for $s \in [t_i, t_{i+1})$ with some discret time grid $0 =t_0 < \cdots < t_N = T$.
\end{definition}

Let us denote 
$$
\Gamma_S := \big\{ \mbox{All strong control terms} \big\},
~~~~
\Gamma_{S,0} := \big\{ \mbox{All piecewise strong control terms} \big\},
$$
and
$$
\Gamma_S(z) ~:=~ \Gamma_S \cap \Gamma(z), ~~\mbox{for}~z \in \R^m \x \R^{\ell}. 
$$
We can then introduce a strong formulation of the optimal control problem under constraints:
$$
V_S(z) ~:=~ \inf_{\gamma \in \Gamma_S(z)} \E^{\P^{\gamma}} \big[ \Phi \big( X^{\gamma}_{\cdot} \big) \big].
$$
As we will see in the technical proof part, one can approximate a weak control by strong control terms.
However, the equality constraints $\E^{\P^{\gamma}}[ \Psi_i(X^{\gamma}_{\cdot})] = z_i$ may not be ensured in the approximation, and for this reason, we are not able to prove that $V_S(z) = V(z)$.
In other words, with the presence of the equality constraints, $V_S(z)$ may lack some regularity in $z$.
Nevertheless, we are able to prove that by considering the convex envelop or lower-semicontinuous envelop of $V_S$, we then obtain $V$.
Let us denote by $V_S^{l.s.c.}$ the biggest lower-semicontinuous function dominated by $V_S$,
and by $V_S^{conv}$ the biggest convex function dominated by $V_S$.

\begin{proposition} \label{prop:VS_V}
Let Assumptions \ref{assum:qualification} and \ref{assum:approximation}  hold true. 
Then,
$$
V(z) 
~=~
V^{conv}_S(z) 
~=~
V^{l.s.c.}_S(z),
~~\mbox{for all}~z \in B_{\eps/2}(0),
$$
where $B_{\eps/2}(0) := \{z \in \R^{m+\ell} ~: \|z\|_1 \le \eps/2\}$,
\end{proposition}
The proof will be completed in Section \ref{subsec:proof_VS_V}.

\paragraph{On Assumption \ref{assum:qualification}} 
Item $\mathrm{(i)}$ in Assumption \ref{assum:qualification} is an integrability condition, which can be reasonably checked for concrete examples.
The qualification condition in Assumption  \ref{assum:qualification} $\mathrm{(ii)}$ turns to be more abstract.
We provide below a more explicit equivalent formulation, which could be easier to check (at least numerically).
Let
$$
E := \{-1, 1\}^m \x \{1\}^{\ell},
~~~\mbox{and}~~
\Theta^+_1 := \Big \{ (\theta_1, \cdots, \theta_{m+\ell}) \in \R_+^{m+\ell} ~: \sum_{i=1}^{m+\ell} \theta_i = 1 \Big\},
$$ 
which is a convex and compact subset of $\R_+^{m+\ell}$.
Let $e \in E$, we denote
$$
e \Psi(\cdot) := \big(e_1 \Psi_1(\cdot), \cdots, e_{m+\ell} \Psi_{m + \ell}(\cdot) \big).
$$

\begin{proposition}
Assumption \ref{assum:qualification} $\mathrm{(ii)}$ holds if and only if,
for each $e \in E$,
\be \label{eq:qualification_conti_time_equiv}
\max_{\theta \in \Theta^+_1} \inf_{\gamma \in \Gamma} \E\big[ \theta \cdot e\Psi(X^{\gamma}_{\cdot}) \big]
=
\inf_{\gamma \in \Gamma} \max_{\theta \in \Theta^+_1}  \E\big[ \theta \cdot e \Psi(X^{\gamma}_{\cdot}) \big]
\le
-\eps,
\ee
and
\be \label{eq:qualification_discrete_time_equiv}
\max_{\theta \in \Theta^+_1} \inf_{\gamma_h \in \Gamma_h} \E\big[ \theta \cdot e \Psi(\Xh^{\gamma_h}_{\cdot}) \big]
=
\inf_{\gamma_h \in \Gamma_h} \max_{\theta \in \Theta^+_1}  \E\big[ \theta \cdot e \Psi(\Xh^{\gamma_h}_{\cdot}) \big]
\le
-\eps.
\ee
\end{proposition}

\begin{remark}
Notice that the optimization problems \eqref{eq:qualification_conti_time_equiv} and \eqref{eq:qualification_discrete_time_equiv} can be computed by the same numerical algorithm that will be described in Section \ref{subsection:numerics}.
\end{remark}

\begin{proof}
We will only provide the equivalence between \eqref{eq:qualification_conti_time} and  \eqref{eq:qualification_conti_time_equiv}. The equivalence between  \eqref{eq:qualification_discrete_time} and  \eqref{eq:qualification_discrete_time_equiv} follows by exactly the same arguments.

First, it is easy to see that the condition \eqref{eq:qualification_conti_time} is equivalent to that for all
$e \in E$, there is some control $\gamma \in \Gamma$ such that
$$
 \E\big[ e_i \Psi_i(X^{\gamma}_{\cdot}) \big] \le - \eps,
 ~~
 \mbox{for all}~
 i=1, \cdots, m + \ell.
$$
Next, we notice that
$$
\max_{i=1, \cdots, m+\ell} c_i ~=~ \max_{\theta \in \Theta^+_1} \big( \theta \cdot c \big),
~~\mbox{for all}~~
c = (c_1, \cdots, c_{m+\ell}) \in \R^{m+\ell},
$$
then \eqref{eq:qualification_conti_time} is equivalent to
$$
\inf_{\gamma \in \Gamma} \max_{\theta \in \Theta^+_1}  \E\big[ \theta \cdot e \Psi(X^{\gamma}_{\cdot}) \big]
\le -\eps,
~~\mbox{for each}~e \in E.
$$
To conclude, it is enough to prove the duality result in \eqref{eq:qualification_conti_time_equiv}.
Let us denote by $\Pc := \{ \P^{\gamma} \circ (X^{\gamma})^{-1} ~: \gamma \in \Gamma \}$,
which is a convex set (see Lemma \ref{lemm:PS_to_PR}).
Notice that $\Theta^+_1$ is a convex compact subset of $\R^{m + \ell}$, then it follows by the classical minimax theorem that
\begin{eqnarray*}
\max_{\theta \in \Theta^+_1} \inf_{\gamma \in \Gamma} \E\big[ \theta \cdot e\Psi(X^{\gamma}_{\cdot}) \big]
&=&
\max_{\theta \in \Theta^+_1} \inf_{\P \in \Pc} \E^{\P}\big[ \theta \cdot e\Psi(X_{\cdot}) \big] \\
&=&
\inf_{\P \in \Pc} \max_{\theta \in \Theta^+_1}  \E^{\P} \big[ \theta \cdot e \Psi(X_{\cdot}) \big]
=
\inf_{\gamma \in \Gamma} \max_{\theta \in \Theta^+_1}  \E\big[ \theta \cdot e \Psi(X^{\gamma}_{\cdot}) \big].
\end{eqnarray*}
\end{proof}

\section{Numerical examples}
\label{sec:Numerical}

\subsection{Numerical algorithm to compute $D_h(0)$} \label{subsection:numerics}

It is not hard to check that $\lambda_h \mapsto d_h(\lambda_h)$ is a concave function, therefore it is natural to compute 
\begin{equation} \label{eq:Dh0_M}
D_h(0) = \max_{\lambda_h \in B_{2M/\eps}(0)} d_h(\lambda_h)
\end{equation}
by the following gradient algorithm,
which is an iterative method (w.r.t.\@ $k$).
Let $(\theta_k)_{k \ge 0}$ be a sequence of positive real numbers.

\begin{itemize}
\item Initialize the problem with $\lambda_0 = 0$.
\item Given $\lambda_k$, solve the problem $d_h(\lambda_k)$ by finding $\hat \gamma_k$ such that
$$
\E^{\mathbb{P}^{\hat{\gamma}_k}} \big[ \Phi\big( \widehat X^{\hat \gamma_k}_{\cdot} \big) + \lambda \cdot \Psi\big( \widehat X^{\hat \gamma_k}_{\cdot} \big) \big]
= d_h(\lambda_k) 
= \inf_{\gamma_h \in \Gamma_h} \E^{\mathbb{P}^{\gamma_h}} \big[ \Phi(\Xh^{ \gamma_h}_{\cdot}) + \lambda^*_h \cdot \Psi(\Xh^{\gamma_h}_{\cdot}) \big].
$$
\item With $(\lambda_k, \hat \gamma_k)$, compute $\lambda_{k+1}$ by
\begin{equation} \label{eq:update_lambda}
\lambda_{k+1} :=  \Pi \big( \lambda_k + \theta_k \nabla d_h(\lambda_k) \big)
~~~\mbox{with}~~
\nabla d_h(\lambda_k) 
:=
\E^{\mathbb{P}^{\hat{\gamma}_k}} \big[ \Psi\big( \widehat X^{ \hat \gamma_k}_{\cdot} \big) \big],
\end{equation}
where $\Pi$ denote the projection operator from $\R^{m+\ell}$ on $B_{2M/\eps}(0) \cap (\R^m \times \R^{\ell}_+)$.
\end{itemize}

Here we commit a slight abuse of notation when writing $\nabla d_h(\lambda_k)$, since $d_h(\lambda_k)$ is not differentiable in general and $\E^{\mathbb{P}^{\hat{\gamma}_k}} \big[ \Psi\big( \widehat X^{ \hat \gamma_k}_{\cdot} \big) \big]$ is only an element of the superdifferential of $d_h$ at $\lambda_k$ (see Lemma \ref{lemma:d_concave}).
By the concavity of $\lambda_h \mapsto d_h(\lambda_h)$ (see again Lemma \ref{lemma:d_concave}), 
	the convergence of $\lambda_k$ to an optimal solution $\hat \lambda_h$ of \eqref{eq:Dh0_M}
	is ensured as soon as the sequence $(\theta_k)_{k \ge 1}$ of  positive real numbers satisfies
	$$
		\sum_{k\ge 0} \theta_k = \infty,
		~~\mbox{and}~~
		\sum_{k\ge 0} \theta_k^2  < \infty,
	$$
	see \cite[Section 5.2.1]{BTN}.
	We also refer to Nesterov \cite[Section 3.2.3]{nesterov} for more discussions on the choices of $\theta_k >0$ as well as the convergence of the gradient algorithm in maximizing a concave function. Some other methods can be implemented for the maximization of $d_h$, in particular cutting-plane methods. We refer the reader to \cite[Chapter XII]{HUL93bis}.

\paragraph{On the choice of $H_h$}

In the purpose of giving a numerical scheme to approximate the continuous time constrained control problem $V(0)$, there are different possible choices for the kernel function $H_h$ satisfying Assumption \ref{assum:approximation}.
	The question goes back to finding numerical approximation schemes for the standard optimal control problem 
	\begin{equation} \label{eq:optimal_control}
		d(\lambda)
		:=
		\sup_{\gamma \in \Gamma} \E^{\P^{\gamma}} \big[ \Upsilon(X^{\gamma}, \lambda) \big],
		~~\mbox{where}~
		dX^{\gamma}_t 
		=
		\mu(t, X^{\gamma}_{t \wedge \cdot}, \alpha^{\gamma}_t) dt 
		+
		\sigma (t, X^{\gamma}_{t \wedge \cdot}, \alpha^{\gamma}_t) d B^{\gamma}_t.
	\end{equation}
	Notice that  the above optimal control problem can be characterized by the HJB equation in the Markovian context, and by the path-dependent PDE in the non-Markovian context.
	In this context, there exist different numerical schemes, such as the finite difference scheme,
	the semi-Lagrangian scheme in Debrabant and Jakobsen \cite{DJ}, the probabilistic scheme of  Fahim, Touzi and Warin \cite{FTW}, etc.
	In the spirit of Kushner and Dupuis \cite{KushnerDupuis}, 
	one can usually interpret a numerical scheme as
	the backward dynamic programming algorithm of a discrete time control problem 
	$$
		d_h(\lambda) := \sup_{\gamma_h \in \Gamma_h} \E^{\gamma_h} \big[ \Upsilon(\widehat X^{\gamma_h}, \lambda) \big],
		~~\mbox{where}~
		X^{\gamma_h}_{t_{k+1}} = X^{\gamma_h}_{t_k} + H_h \big(t_k, \widehat X^{\gamma_h}_{t_k \wedge \cdot}, \alpha^{\gamma_h}_{t_k}, U^{\gamma}_{t_{k+1}} \big),
	$$
	with appropriate choice of $H_h$.
	In practice, one can usually describe and implement the numerical schemes for optimal control problem \eqref{eq:optimal_control} without giving the explicit expression of $H_h$.
	The kernel function $H_h$ is only used in the technical proof of the convergence.
	Below, we provide a detailed discussion on $H_h$ for the finite difference scheme, and also
	refer to Section 4 of Ren and Tan \cite{RenTan} and Possama{\"i} and Tan \cite{PT} for more detailed discussions on the kernel function $H_h$ for different numerical schemes.

\paragraph{On the resolution of the (unconstrained) control problem $d_h(\lambda_k)$}

	For completeness, we detail the practical computation of $d_h(\lambda_k)$ and $\nabla d_h(\lambda_k)$ when the finite-difference scheme is employed.
	For simplicity of notations, let us stay in the one-dimensional ($d=1$) Markovian case, i.e. $(\mu, \sigma, \Phi, \Psi)$ depends on $X^{\gamma}_t$ in place of the whole path $X^{\gamma}_{t \wedge \cdot}$.

	The finite difference scheme to solve the control problem \eqref{eq:optimal_control} consists first in discretizing $[0,T] \x \R$ into the discrete grid $(t_k, x_i)_{k = 1, \cdots, n, ~i \in \Z}$,
	where $t_k = k \Delta t$, $x_i  := i \Delta x$ satisfies $t_0 = 0$ and $t_n = T$.
	Next, one sets the terminal condition
	$$
		v(t_n, x_i) ~:=~ \Upsilon(x_i, \lambda) = \Phi(x_i) + \lambda \cdot \Psi(x_i),
	$$
	and then computes $v(t_k, \cdot) = v^k_i$ by the backward iteration:
	$$
		v^k_i
		=
		v^{k+1}_i
		+
		\Delta t
		\max_{a \in A} \Big( 
			\mu(t_k, x_i, a) \frac{v^{k+1}_{i+1} - v^{k_1}_i }{\Delta x}
			+ 
			\frac12 \sigma^2(t_k, x_i, a) \frac{v^{k+1}_{i+1} - 2 v^{k+1}_i + v^{k+1}_{i-1}}{\Delta x^2}
		\Big).
	$$
	By direct computation, one can rewrite the backward iteration as
	\begin{equation} \label{eq:DF_vki}
		v^k_i 
		~=~ 
		\max_{a \in A} \Big(
			p_+^{k,i}(a) ~ v^{k+1}_{i+1} + p_-^{k,i}(a) ~v^{k+1}_{i-1} + \big( 1- p_+^{k,i}(a) - p_-^{k,i}(a) \big) ~v^{k+1}_i 
		\Big),
	\end{equation}
	where
	$$
		p_+^{k,i}(a) = \mu(t_k, x_i, a) \frac{\Delta t}{\Delta x} + \frac12 \sigma^2(t_k,  x_i, a) \frac{\Delta t}{\Delta x^2},
		~~~~
		p_-^{k,i}(a) = \frac12 \sigma^2(t_k,  x_i, a) \frac{\Delta t}{\Delta x^2}.
	$$
	Let us denote one optimal choice in \eqref{eq:DF_vki} by
	$\hat a_{k,i}$, and denote $\hat p^{k,i}_+ = p^{\hat a_{k,i}}_+$ and  $\hat p^{k,i}_- = p^{\hat a_{k,i}}_-$. 
	Under appropriate conditions on $\mu, \sigma, \Delta t, \Delta x$, one can ensure that $p_+^{k,i}(a)$, $p_-^{k,i}(a)$ and $(1- p_+^{k,i}(a) - p_-^{k,i}(a))$ are all in $[0,1]$.
	
	Following Kushner and Dupuis \cite{KushnerDupuis},
	one can interpret the backward iteration \eqref{eq:DF_vki} as an optimization problem on
	a controlled Markov chain system,
	where at each node $(t_k, x_i)$, with control $a$, the Markov chain jumps from
	$x_i$ to $x_{i+1}$ with probability $p^{k,i}_+(a)$, and from $x_i$ to $x_{i-1}$ with probability $p^{k,i}_-(a)$, and stays at $x_i$ with probability $1 - p^{k,i}_+(a) - p^{k,i}_-(a)$.
	Therefore, the kernel function $H_h$ could be constructed explicitly such that, with r.v. $U \sim \Uc[0,1]$, 
	$$
		H_h(t_k, x_i, a, U) = 
		\begin{cases}
		\Delta x, &\mbox{with probability}~ p^{k,i}_+(a) ,\\
		-\Delta x, &\mbox{with probability}~p^{k,i}_-(a) ,\\
		0, &\mbox{with probability}~ 1 - p^{k,i}_+(a)  - p^{k,i}_-(a) .
		\end{cases}
	$$
	By direct computation, one can check that
$H_h$ satisfies Assumption \ref{assum:approximation}$\mathrm{(ii)}$.
	Moreover, with the optimal control $\hat a_{k,i}$ at each node $(t_k, x_i)$, 
	and the induced optimal probability function $(\hat p^{k,i}_+, \hat p^{k,i}_-, 1- \hat p^{k,i}_+ - \hat p^{k,i}_-)$,
	one obtains the optimal controlled Markov chain denoted by $\hat \gamma$.
	It follows that one can compute the subgradient $\nabla d_h (\lambda) = \E^{\P^{\hat \gamma}}\big[ \Psi(X^{\hat \gamma}_T) \big]$ in \eqref{eq:update_lambda} (recall that $\Psi$ is assumed to be Markovian for simplicity) using the tower property of conditional expectation.
	More precisely, we set the terminal condition $u(t_n, x_i) := \Psi(x_i)$ and 
	then compute $u(t_k, x_i)$ with the backward iteration
	\begin{equation} \label{eq:computation_Dd_h}
		u(t_k, x_i) 
		=
		\hat p^{k,i}_+ ~u(t_k, x_{i+1}) + \hat p^{k,i}_- ~u(t_k, x_{i-1}) + \big( 1 - \hat p^{k,i}_+ - \hat p^{k,i}_- \big) ~u(t_k, x_{i}),
	\end{equation} 
	and finally one obtains $\nabla d_h(\lambda) = u(t_0, x_0)$.
	
	For many other numerical schemes, such as the semi-Lagrangian scheme in Debrabant and Jakobsen \cite{DJ}, the probabilistic scheme of  Fahim, Touzi and Warin \cite{FTW}, etc., one can write the numerical method in form of \eqref{eq:DF_vki} with some controlled probability $p(\cdot, a)$, and with the optimal controlled probability $p(\cdot, \hat a(\cdot) )$, one can compute $\nabla d_h(\lambda)$ as in \eqref{eq:computation_Dd_h}.
	
	Notice that these schemes can generally be described and implemented as in the backward iteration \eqref{eq:DF_vki} and \eqref{eq:computation_Dd_h} without giving an explicit expression of the kernel function $H_h$, which serves only in the technical convergence proof.	

\subsection{Constrained linear-quadratic problems}

In this section we illustrate our theoretical results by two simple numerical examples with a linear quadratic structure,
so that the reference value of the problem can be computed explicitly.
Concretely, we set $A = \R$, $\mu(t,x,a) = a$, $\sigma(t,x,a) \equiv 1$,
so that for all $\gamma \in \Gamma$, $X^{\gamma}$ is given by
\begin{equation} \label{eq:sde1d}
X^{\gamma}_t = x_0 + \int_0^t  \alpha^{\gamma}_s ds + B^{\gamma}_t,
~~\mbox{where}~
\E^{\P^{\gamma}} \Big[ \int_0^T \big( \alpha^{\gamma}_s \big)^2 ds \Big] < \infty.
\end{equation}
Given $T>0$ and $(a_T,b_T,c_T) \in \R^3$, we define the value function
\begin{equation*}
V(x_0; T,a_T,b_T,c_T) 
~:=~ 
\inf_{\gamma \in \Gamma}
\mathbb{E} \Big[
\int_0^T \frac{1}{2} (\alpha^{\gamma}_t)^2 dt + \frac{1}{2} a_T (X_T^{\gamma})^2 + b_T X_T^{\gamma} + c_T
\Big].
\end{equation*}
With the above linear quadratic structure, one can check that,
if $1+ Ta_T > 0$, then
\begin{equation*}
V(x_0; T,a_T,b_T,c_T)
=
\frac{1}{2} a_0 x_0^2 + b_0 x_0 + c_0,
\end{equation*}
where
\begin{equation} \label{eq:riccati}
a_0= \frac{a_T}{1+ Ta_T}, \quad
b_0= \frac{b_T}{1+ Ta_T}, \quad
c_0= c_T - \frac{Tb_T^2}{2(1+Ta_T)} + \frac{1}{2} \ln(1+Ta_T).
\end{equation}
If $1+T a_T \le 0$, then $V(T,a_T,b_T,c_T)= - \infty$.
\vspace{0.5em}

\emph{Example 1.}
We consider the constrained optimization problem
\begin{equation*}
\inf_{\gamma \in \Gamma} \ \mathbb{E}^{\P^{\gamma}} \Big[ \int_0^1 \frac{1}{2} (\alpha^{\gamma}_t)^2 dt + (X_T^{\gamma})^2 \Big],
\quad
\text{subject to: } \mathbb{E}^{\P^{\gamma}} \big[ -X_T^{\gamma} + 1 \big] = 0,
\end{equation*}
where $X^{\gamma}$ follows \eqref{eq:sde1d} with $x_0 = 0$.
The dual criterion $d(\lambda)$ can be explicitly calculated with formula \eqref{eq:riccati}, applied with $a_1= 2$, $b_1= - \lambda$, and $c_1= \lambda$. 
We obtain
$
d(\lambda)= -\frac{1}{6} \lambda^2 + \lambda + \frac{1}{2} \ln(3).
$
The function $d$ reaches its maximum $\frac{3}{2} + \frac{1}{2} \ln(3)$ at $\lambda^*= 3$.
\vspace{0.5em}

\emph{Example 2.}
We next consider the problem
\begin{equation*}
\inf_{\gamma \in \Gamma} \ \mathbb{E}^{\P^{\gamma}} \Big[ \int_0^1 \frac{1}{2} (\alpha^{\gamma}_t)^2 dt + (X_T^{\gamma})^2 \Big],
\quad
\text{subject to: } \mathbb{E}^{\P^{\gamma}} \big[ (X_T^{\gamma} -1)^2 \big]  - \frac{1}{2} \le 0,
\end{equation*}
where $X^\gamma$ follows \eqref{eq:sde1d} with $x_0= 0$.
The dual criterion, obtained by applying \eqref{eq:riccati} with $a_1= 2+2\lambda$, $b_1= -2\lambda$, $c_1= \frac{1}{2} \lambda$, is given by
\begin{equation*}
d(\lambda)=
\begin{cases}
\begin{array}{ll}
-\frac{2\lambda^2}{3+ 2 \lambda} + \frac{1}{2} \lambda + \frac{1}{2} \ln(3+2\lambda), & \text{if $\lambda \geq -3/2$}, \\
- \infty, & \text{otherwise.}
\end{array}
\end{cases}
\end{equation*}
The function $d$ reaches its maximum (approximately) $0.91$ at $\lambda^* = -1 + \frac{1}{2} \sqrt{19} \approx 1.12$.

We have implemented the numerical procedure as described in Section \ref{subsection:numerics}, combined with two different discretization schemes:
\begin{itemize}
\item the finite-difference scheme (with $\Delta t = h >0$ and $\Delta x= \sqrt{\Delta t/0.1}$).
\item the semi-Lagrangian scheme (with $\Delta t= \Delta x = h$).
\end{itemize}
The results are presented on Figures \ref{fig:numerics} (Example 1) and \ref{fig:numerics2} (Example 2), with a logarithmic scale.
The dual problem has been solved at a high precision, so that $|\nabla d_h(\lambda)| < 10^{-7}$.

We observe that for both examples, empirical rates of convergence can be observed for $|V(0)-V_h(0)|$:
\begin{itemize}
\item equal to $1/2$ for the finite-difference scheme
\item equal to $1$ for the semi-Lagrangian scheme.
\end{itemize}
A rate equal to $1/2$ can also be observed for $|\lambda_k-\lambda^*|$ for the finite-difference scheme, for both examples. For Example 1, the semi-Lagrangian scheme was able to find the exact value of $\lambda^*$. For Example 2, the error $|\lambda_k-\lambda^*|$ converges with rate 1 in the case of the semi-Lagrangian scheme.
Let us emphasize that our theoretical findings only concern the convergence of $|V(0)-V_h(0)|$. 
	 It is still an open question to us on how to prove the convergence or provide a convergence rate of $\lambda_h$ as $h \longrightarrow 0$.
	Let us also mention that under the actual qualification condition, uniqueness of the solution $\lambda^*$ to the dual problem is not guaranteed in general.

\begin{figure}[htb]
\begin{center}
\includegraphics[trim = 1cm 8.5cm 0cm 8cm, clip, scale=0.28]{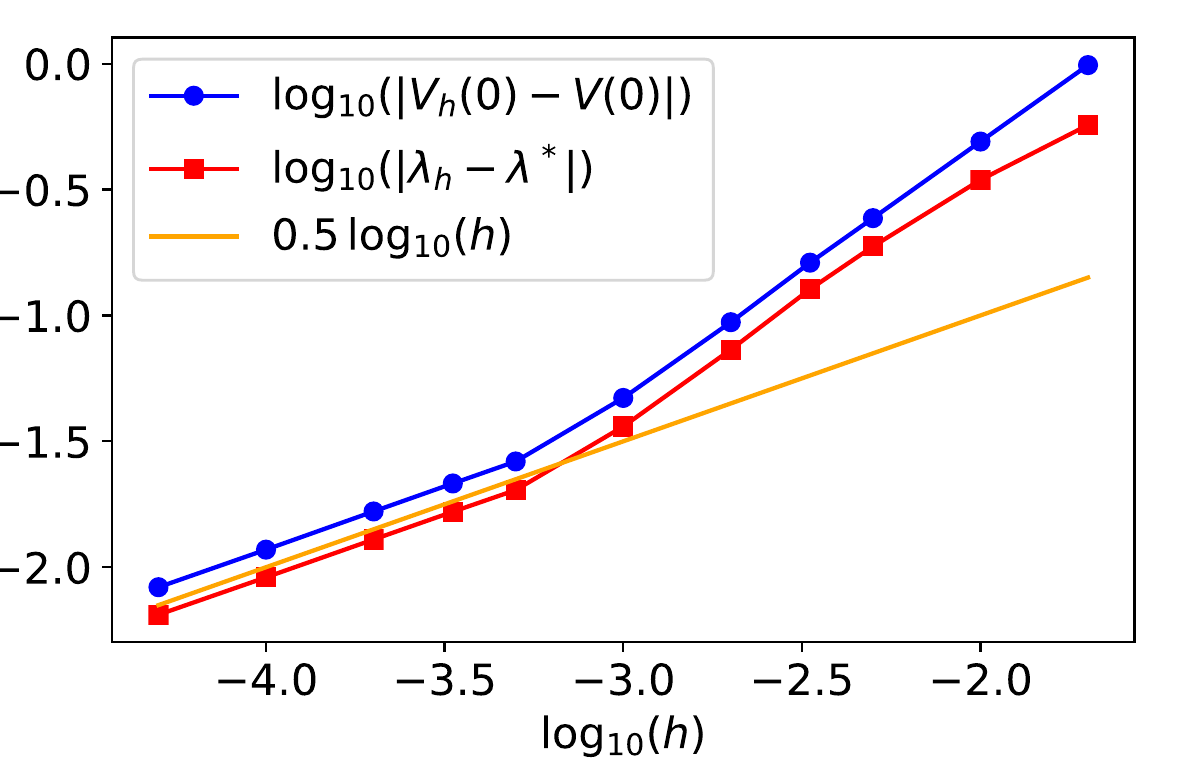}
\hspace{3mm}
\includegraphics[trim = 1cm 8.5cm 0cm 8cm, clip, scale=0.28]{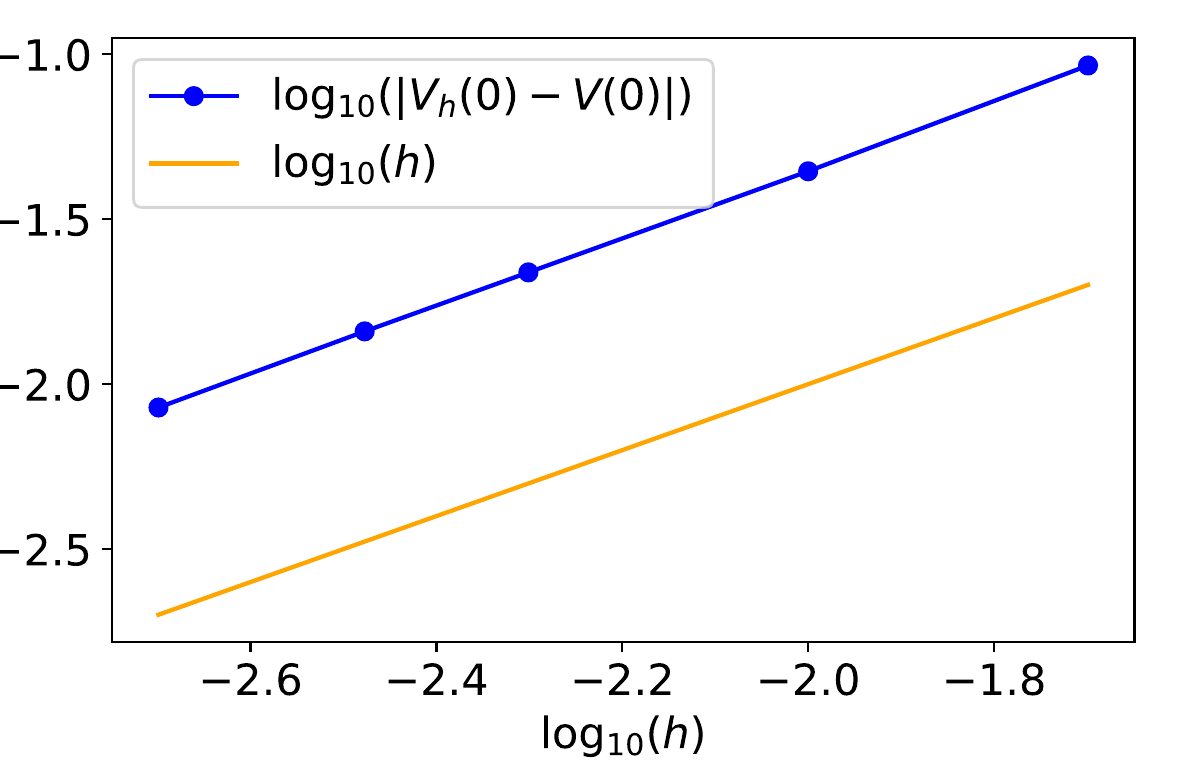}
\caption{Approximation errors for Example 1 (left: finite-difference scheme, right: semi-Lagrangian scheme).}
\label{fig:numerics}
\end{center}
\end{figure}

\begin{figure}[htb]
\begin{center}
\includegraphics[trim = 1cm 8.5cm 0cm 8cm, clip, scale=0.28]{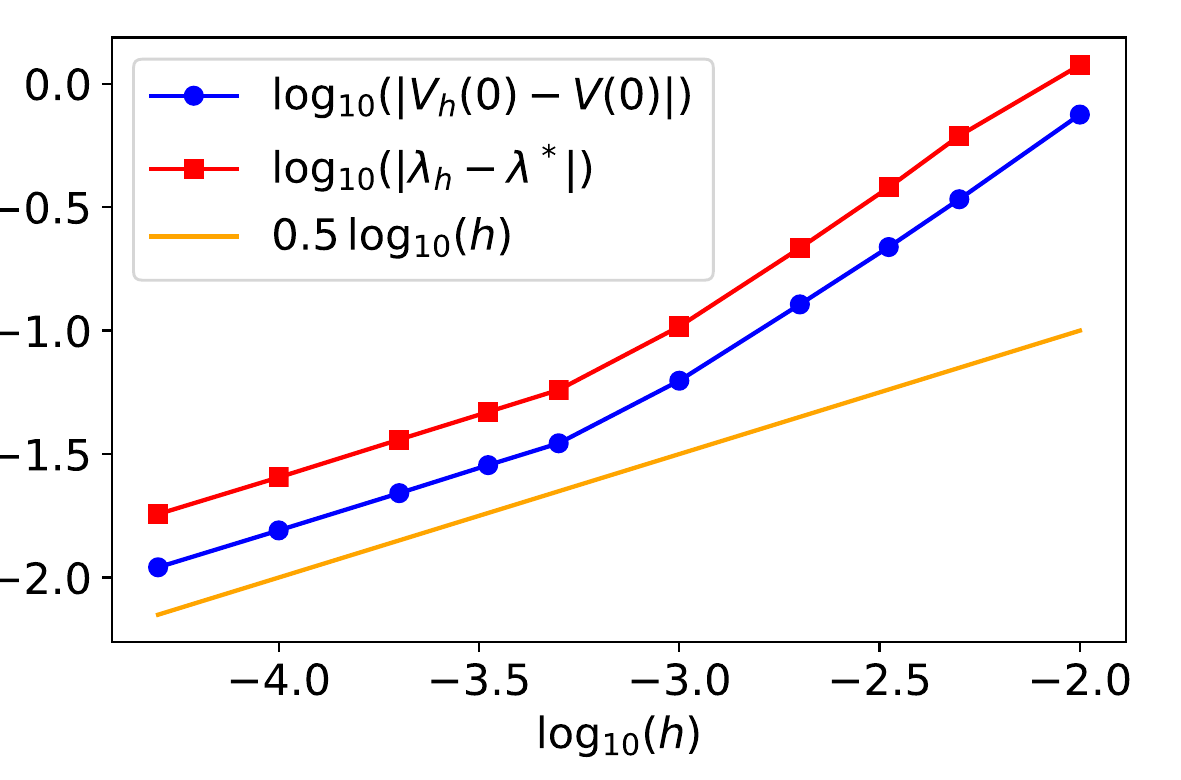}
\hspace{3mm}
\includegraphics[trim = 1cm 8.5cm 0cm 8cm, clip, scale=0.28]{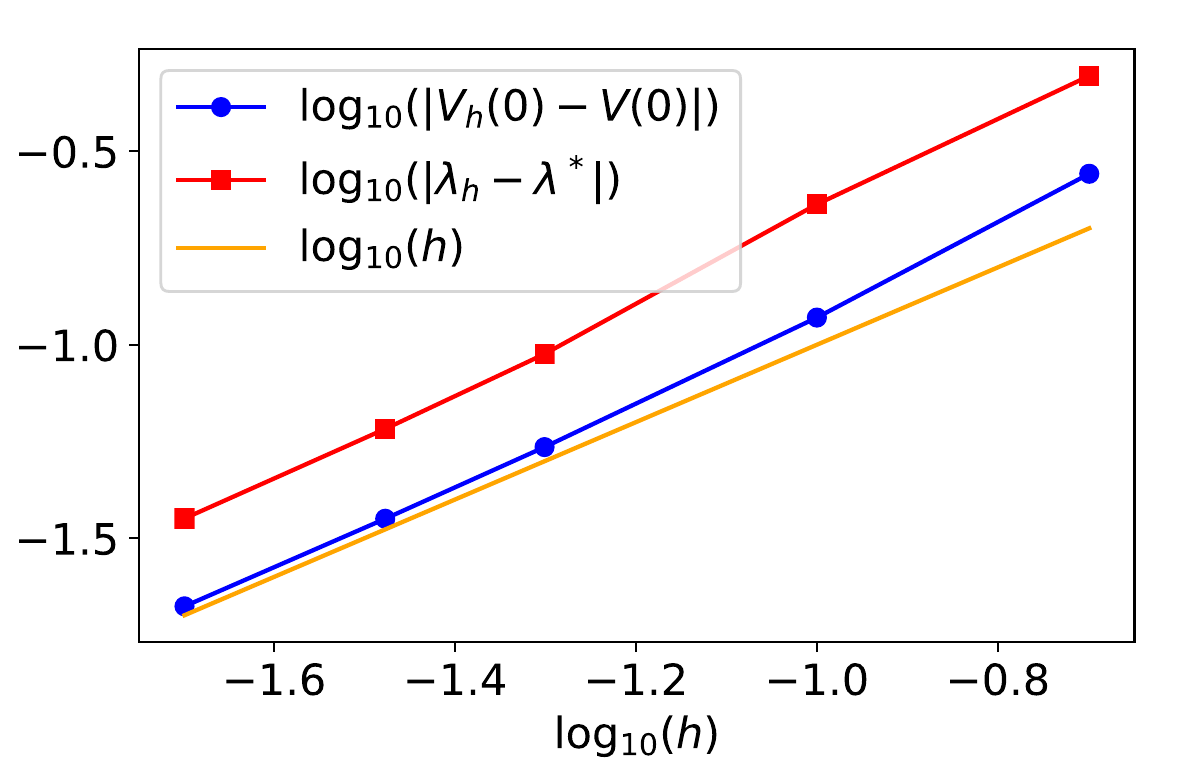}
\caption{Approximation errors for Example 2 (left: finite-difference schem, right: semi-Lagrangian scheme).}
\label{fig:numerics2}
\end{center}
\end{figure}


\subsection{A super-replication problem}

We next provide a numerical example in the context of Example \ref{exam:vol_uncertain} with the following parameters:
$S_0 = 1$, $\underline \sigma = 0.04$, $\overline \sigma = 0.3$, $T=1$, $\Psi_i(S_{\cdot}) = (S_T - K_i)_+$, for $i=1,2$ with $K_1 = 0.9$, $K_2 = 1.1$, $c_1 = 0.13$, and $c_2 = 0.04$, $\Phi(S_{\cdot}) = (\int_0^T S_t dt - K)_+$ with $K =1$.

Then the dual problem is
$$
	\inf_{\lambda_1 \in \R, \lambda_2 \ge 0} \sup_{\gamma \in \Gamma} 
	\E \Big[ \Phi(S^{\gamma}) - \lambda_1\big( \Psi_1(S^{\gamma}_T) - c_1 \big)- \lambda_2 \big( \Psi_1(S^{\gamma}_T) - c_2 \big)
	\Big],
$$
where $\Gamma$ consists of the set of all models with dynamic
$$
	dS^{\gamma}_t = \sigma^{\gamma}_t  S^{\gamma}_t  dW^{\gamma}_t,
	~~~
	\underline \sigma \le \sigma^{\gamma}_t \le \overline \sigma,
	~\mbox{for all}~t \in [0,T]~\mbox{a.s.}
$$ 

For the numerical resolution of the problem, we have performed the change of variable $X_t= \ln(S_t)$. 
An additional state variable $Y_t= \int_0^t S_{\theta} d {\theta} = \int_0^t \exp(X_{\theta}) d {\theta}$ has been introduced, in order to deal with the cost $\Phi$.
We have employed the semi-Lagrangian scheme, with $\Delta t= \Delta x= \Delta y=h$. 
For the resolution of the (discretized) dual problem, we have employed the cutting-plane method described in \cite[Chapter XII]{HUL93bis} which turned out to be very effective.
We report below the obtained results for various values of $h$. A clear convergence of $V_h(0)$ can be observed.
\begin{equation*}
\begin{array}{|c|c|c|c|} \hline
h & V_h(0) & \lambda_{1,h} & \lambda_{2,h} \\ \hline
1/5 & 0.0188 & 9.78 & 0 \\
1/6 & 0.0290 & 2.086 & 0 \\
1/10 & 0.0584 & 0.513 & 0.340 \\
1/20 & 0.0615 & 0.741 & 0.213 \\
1/30 & 0.0630 & 0.537 & 0.226 \\
1/50 & 0.0635 & 0.385 & 0.200 \\ \hline
\end{array}
\end{equation*}

\section{Proofs}
\label{sec:Proofs}

\subsection{An abstract duality result}
\label{subsec:abstract_duality}

We consider in this subsection an abstract formulation of the optimal control problem with expectation constraints. 
This formulation will facilitate the presentation of the so called Lagrange relaxation approach, which is at the core of the proof of Theorem \ref{thm:duality}. Let us refer the reader to \cite[Chapter XII]{HUL93} for a detailed presentation of this classical approach. For completeness, we first recall some basic definitions in convex analysis.
For a convex function $f \colon \R^{m + \ell} \rightarrow \R$, the subdifferential $\partial f(x)$ is defined by
\begin{equation*}
\partial f(x)
= \big\{ \lambda \in \R^{m+\ell} ~: f(y) \geq f(x) + \langle \lambda, y-x \rangle, \ \forall y \in \R^{m + \ell} \big\}.
\end{equation*}
The superdifferential of a concave function is defined similarly, by changing the inequality sign in the above inequality. The conjugate $f^* \colon \R^{m+\ell} \rightarrow \R$ of the function $f$ is defined by
\begin{equation*}
f^*(\lambda)
=
\sup_{x \in \R^{m+\ell}} \Big( \langle \lambda, x \rangle - f(x) \Big).
\end{equation*}

All along this subsection, we consider an abstract measurable space $(\Om, \Fc)$, which we equip with a convex subset $\mathcal{P}$ of $\mathcal{P}(\Omega)$
and some random variables $\xi$ and $\zeta_1,...,\zeta_{m+\ell}$.
We make the following assumptions:
\begin{itemize}
\item[(A1)] For all $\mathbb{P} \in \mathcal{P}$ and for all $i=1,...,m + \ell$, we have
$\mathbb{E}^{\mathbb{P}} [ |\xi | ] < \infty$ and
$\mathbb{E}^{\mathbb{P}} [ | \zeta_i | ] < \infty$;
\item[(A2)] $ M := \sup_{\P \in \Pc} \big| \mathbb{E}^{\mathbb{P}} [ \xi ] \big| < \infty$;
\item[(A3)] There exists $\varepsilon > 0$ such that the following inclusion holds true:
\begin{equation*}
B_\varepsilon(0) \subseteq
\big\{ z \in \R^{m + \ell}  ~: \exists \mathbb{P} \in \mathcal{P}, \, \mathbb{E}^{\mathbb{P}} [ \zeta ] + z \in \{ 0_m \} \times \R_-^{\ell} \big\},
\end{equation*}
where $B_\varepsilon(0)$ denotes the ball of radius $\varepsilon$ and center 0 for the supremum norm.
\end{itemize}

We aim at analyzing the following problem, referred to in this subsection as primal problem:
\begin{equation} \label{eq:primal}
V(0) := \inf_{\mathbb{P} \in \mathcal{P}} \mathbb{E}^{\mathbb{P}} [ \xi ], \quad \text{subject to: }
\begin{cases}
\begin{array}{l}
\mathbb{E}^{\mathbb{P}} [ \zeta_i ] = 0, \quad  \forall i=1,...,m, \\
\mathbb{E}^{\mathbb{P}} [ \zeta_i ] \leq 0, \quad \forall i=m+1,...,m+\ell.
\end{array}
\end{cases}
\end{equation}
Observe that this is a convex problem, since $\mathcal{P}$ is assumed to be convex and the expectation is linear with respect to the involved probability measure.
We next introduce the dual criterion $d \colon \R^{m + \ell} \rightarrow \R$, defined as follows:
\begin{equation*}
d(\lambda) := \inf_{\mathbb{P} \in \mathcal{P}} \mathcal{L}(\mathbb{P},\lambda),
~~\mbox{where}~\mathcal{L}(\mathbb{P},\lambda) := \mathbb{E}^{\mathbb{P}} [ \xi ]
+ \langle \lambda, \mathbb{E}^{\mathbb{P}} [\zeta] \rangle
~\mbox{is the Lagrangian}.
\end{equation*}
Finally, the dual problem to \eqref{eq:primal} is given by
\begin{equation} \label{eq:dual}
D(0) := \sup_{\lambda \in \R^{m + \ell}} d(\lambda), \quad \text{subject to: } \lambda \in \{ 0_m \} \times \R_+^{\ell}.
\end{equation}
The main result of the subsection is the following proposition.

\begin{proposition} \label{prop:duality}
The primal problem \eqref{eq:primal} and the dual problem \eqref{eq:dual} have the same value. Moreover, the set of solutions to the dual problem is non-empty, coincide with $\partial V(0)$, and is bounded. 
More precisely, any dual optimizer $\lambda^*$ satisfies
\begin{equation} \label{eq:estimate_lambda}
\| \lambda \|_1 := \sum_{i=1}^{m + \ell} |\lambda_i| \leq \frac{2M}{\varepsilon}.
\end{equation} 
\end{proposition}

\begin{proof}
The key idea of the approach consists in parametrizing Problem \eqref{eq:primal} by introducing a variable $z$ in the constraints. We introduce the function $V \colon \R^{m + \ell} \rightarrow \R$, defined as follows:
\begin{equation} \label{eq:Abstract_V}
V(z)=
\inf_{\mathbb{P} \in \mathcal{P}} \, \mathbb{E}^{\mathbb{P}} [ \xi ], \quad \text{subject to: }
\mathbb{E}^{\mathbb{P}} [ \zeta ] + z  \in \{ 0_m \} \times \R_-^{\ell}. \\
\end{equation}

\emph{Step 1: The function $V$ is convex and finite in a neighborhood of 0}.
Take $z^{(1)},z^{(2)} \in \R^{m+\ell}$, and $\mathbb{P}_1, \mathbb{P}_2 \in \mathbb{P}$ such that $\mathbb{E}^{\mathbb{P}_i} [ \zeta ] + z^{(i)} \in \{ 0_m \} \times \R_-^{\ell}$, $i=1,2$. Let $\theta \in [0,1]$. Let us set $z = \theta z^{(1)} + (1-\theta)z^{(2)}$ and $\mathbb{P}= \theta \mathbb{P}_1 + (1-\theta) \mathbb{P}_2$. By convexity of $\mathcal{P}$, we have that $\mathbb{P} \in \mathcal{P}$. Moreover,
\begin{equation*}
\mathbb{E}^{\mathbb{P}} [ \zeta ] + z
= \theta \big( \mathbb{E}^{\mathbb{P}_1} [\zeta ] + z^{(1)} \big) + (1-\theta) \big( \mathbb{E}^{\mathbb{P}_2} [ \zeta ] + z^{(2)} \big) \in \{ 0_m \} \times \R_-^{\ell}.
\end{equation*}
Therefore,
\begin{equation*}
V(z) \leq \mathbb{E}^{\mathbb{P}} [ \xi ]
\leq \theta \mathbb{E}^{\mathbb{P}_1} [ \xi ] + (1-\theta) \mathbb{E}^{\mathbb{P}_2} [ \xi ].
\end{equation*} 
Minimizing the right-hand side with respect to $\mathbb{P}_1$ and $\mathbb{P}_2$, we finally obtain that $V(z) \leq \theta V(z^{(1)}) + (1-\theta) V(z^{(2)})$, which concludes the proof of convexity.
By Assumption (A2), the function $V$ is bounded from below by $-M$. By Assumption (A3), we have that for all $z \in B_\varepsilon(0)$, the optimization problem associated with $V(z)$ is feasible, thus $V(z) < \infty$. This proves that $V$ is finite in a neighbourhood of 0.

\emph{Step 2: Calculation of the conjugate function $V^*$.}
For $\lambda \in \R^{m + \ell}$, we have
\begin{align*}
V^*(\lambda)
= \ & \sup_{z \in \R^{m + \ell}} \big( \langle \lambda, z \rangle - V(z) \big) \\
= \ & \sup_{z \in \R^{m + \ell}, \, \mathbb{P} \in \mathcal{P}} \
\langle \lambda, z \rangle - \mathbb{E}^{\mathbb{P}} [ \xi ], \quad \text{subject to: }
\mathbb{E}^{\mathbb{P}}[ \zeta ] + z \in \{ 0_m \} \times \R_-^{\ell}.
\end{align*}
We make the changes of variable $z'= \mathbb{E}^{\mathbb{P}}[ \zeta ] + z$. Observing that
\begin{equation*}
\langle \lambda, z \rangle - \mathbb{E}^{\mathbb{P}} [\xi ]
= \langle \lambda, z' \rangle - \mathcal{L}(\mathbb{P},\lambda),
\end{equation*}
we deduce that
\begin{equation*}
V^*(\lambda)= \sup_{z' \in \{ 0_m \} \times \R_-^{\ell} } \langle \lambda, z' \rangle - \inf_{\mathbb{P} \in \mathcal{P}} \mathcal{L}(\mathbb{P},\lambda).
\end{equation*}
For all $\lambda \in \R^{m + \ell}$, we have
\begin{equation*}
\sup_{z' \in \{ 0_m \} \times \R_-^{\ell} } \langle \lambda, z' \rangle
= \begin{cases}
\begin{array}{cl}
0, & \text{ if $\lambda \in \R^m \times \R_+^{\ell}$,} \\
+ \infty, & \text{ otherwise.}
\end{array}
\end{cases}
\end{equation*}
Therefore,
\begin{equation*}
V^*(\lambda)= \begin{cases}
\begin{array}{cl}
- \inf_{\mathbb{P} \in \mathcal{P}} \mathcal{L}(\mathbb{P},\lambda) = - d(\lambda), & \text{ if $\lambda \in \R^m \times \R_+^{\ell}$}, \\
+ \infty, & \text{ otherwise.}
\end{array}
\end{cases}
\end{equation*}

\emph{Step 3: Duality result.} As a consequence of \cite[Proposition 2.108]{BS00}, the mapping $V$ is continuous at $0$, since it is convex and finite in a neighborhood of 0. It follows with \cite[Proposition 2.126]{BS00} that $\partial V(0)$ is non-empty. It follows further with \cite[Proposition 2.118]{BS00} that $V(0)= V^{**}(0)$. Finally, the equivalence in \cite[Relation 2.232]{BS00} ensures that for all $\lambda \in \R^{m + \ell}$,
\begin{equation*}
V(0) + V^*(\lambda) = \langle \lambda, 0 \rangle = 0
\Longleftrightarrow \lambda \in \partial V(0).
\end{equation*}
Therefore, we have
\begin{equation*}
V(0)= V^{**}(0) = \sup_{\lambda \in \R^{m + \ell}} \langle 0, \lambda \rangle - V^*(\lambda)
= \sup_{\lambda \in \R^m \times \R^\ell_+} d(\lambda),
\end{equation*}
proving that the primal and the dual problems have the same value. 
For all $\lambda \in \partial V(0)$, we have $V(0)= - V^*(\lambda)$, thus $V^*(\lambda)$ is finite, and necessarily, we have $\lambda \in \R_m \times \R_+^{\ell}$ and $V^*(\lambda)= - d(\lambda)$, which proves that $\lambda$ is a solution to the dual problem.
Conversely, if $\lambda$ is a solution to the dual problem, then we have $V(0)= d(\lambda)= - V^*(\lambda)$ and therefore $\lambda \in \partial V(0)$.

\emph{Step 4: Boundedness of the set of dual solutions.}
Let $\lambda^* \in \partial V(0)$. Let $z$ be defined by $z_i= \text{sign}(\lambda^*_i) \varepsilon$.
By Assumptions (A2) and (A3), we have $V(z) \leq M$ and $V(0) \geq - M$.
We have
\begin{equation*}
M \geq V(z) \geq V(0) + \langle \lambda^*, z \rangle 
\geq
 - M + \varepsilon  \sum_{i=1}^{m + \ell} | \lambda^*_i |
= 
- M + \varepsilon \| \lambda^* \|_1.
\end{equation*}
Then it follows that \eqref{eq:estimate_lambda} holds true.
\end{proof}

As consequence of Proposition \ref{prop:duality}, we obtain the following optimality conditions.

\begin{corollary} \label{coro:charact_dual_optimizer_abstr}
For all solutions $\mathbb{P}^*$ to the primal problem \eqref{eq:primal} and for all solutions $\lambda^*$ to the dual problem \eqref{eq:dual}, one has
\begin{equation} \label{eq:opti_cond1}
\mathcal{L}(\mathbb{P}^*,\lambda^*)= \min_{\mathbb{P} \in \mathcal{P}} \mathcal{L}(\mathbb{P},\lambda^*).
\end{equation}
Moreover, for each $i=m+1,...,m+\ell$, one has $ \lambda^*_i = 0$ whenever
$\mathbb{E}^{\mathbb{P}^*} [ \zeta_i ] < 0$,
i.e.\@ $\langle \lambda^*, \mathbb{E}^{\mathbb{P}^*} [ \zeta] \rangle= 0$.
\end{corollary}

\begin{proof}
Since $\mathbb{E}^{\mathbb{P}^*} [ \zeta ] \in \{ 0_m \} \times \R_-^{\ell}$ and $\lambda^* \in \R^m \times \R_+^{\ell}$, we have
\begin{equation*}
\mathcal{L}(\mathbb{P}^*,\lambda^*)
= \mathbb{E}^{\mathbb{P}^*}[ \xi ] + \langle \lambda^*, \mathbb{E}^{\mathbb{P}^*} [ \zeta] \rangle
\leq \mathbb{E}^{\mathbb{P}^*}[ \xi ].
\end{equation*}
By Proposition \ref{prop:duality}, $\lambda^* \in \partial V(0)$. Therefore,
\begin{equation*}
\mathcal{L}(\mathbb{P}^*,\lambda^*)
\leq V(0)
= - V^*(\lambda^*)
= d(\lambda^*)
= \inf_{\mathbb{P} \in \mathcal{P}} \mathcal{L}(\mathbb{P},\lambda^*).
\end{equation*}
Of course, since $\mathbb{P}^* \in \mathcal{P}$, we also have that $\mathcal{L}(\mathbb{P}^*,\lambda^*) \geq \inf_{\mathbb{P} \in \mathcal{P}} \mathcal{L}(\mathbb{P},\lambda^*)$. As a consequence, all the above inequalities are equalities. Equality \eqref{eq:opti_cond1} is proved, and we also have that $\langle \lambda^*, \mathbb{E}^{\mathbb{P}^*} [ \zeta] \rangle= 0$, from which the complementarity condition easily follows.
\end{proof}

The necessary optimality conditions of Corollary \ref{coro:charact_dual_optimizer_abstr} are also sufficient optimality conditions.

\begin{lemma} \label{lemma:sufficiency}
Let $\mathbb{P}^* \in \mathcal{P}$. Assume that there exists $\lambda^* \in \R^m \times \R_+^{\ell}$ such that $\langle \lambda^*, \mathbb{E}^{\mathbb{P}^*} [ \zeta] \rangle= 0$ and such that
\begin{equation*}
\mathcal{L}(\mathbb{P}^*,\lambda^*)= \min_{\mathbb{P} \in \mathcal{P}} \mathcal{L}(\mathbb{P},\lambda^*).
\end{equation*}
Then $\mathbb{P}^*$ is a solution to \eqref{eq:primal}.
\end{lemma}

\begin{proof}
Let $\mathbb{P}$ be feasible for \eqref{eq:primal}. Then the following inequalities and equality can be easily verified:
$
\mathbb{E}^{\mathbb{P}}[\xi]
\geq \mathcal{L}(\mathbb{P},\lambda^*)
\geq \mathcal{L}(\mathbb{P}^*,\lambda^*)
= \mathbb{E}^{\mathbb{P}^*}[\xi],
$
which proves the lemma.
\end{proof}

We give now a property of the superdifferential of $d$, which in particular can be used in a numerical algorithm to solve the dual problem.

\begin{lemma} \label{lemma:d_concave}
The map $d$ is concave. Moreover, for all $\lambda \in \R^{m + \ell}$, for all $\mathbb{P} \in \mathcal{P}$ such that $d(\lambda)= \mathcal{L}(\mathbb{P},\lambda)$, the vector $z:= \mathbb{E}^{\mathbb{P}} [ \zeta ]$ lies in the superdifferential of $d$ at $\lambda$.
\end{lemma}

\begin{proof}
The Lagrangian $\mathcal{L}$ is affine with respect to $\lambda$, therefore concave. The mapping $d(\lambda)$ is expressed as an infimum of concave functions, thus it is concave. Let $\lambda \in \R^{m + \ell}$. Let $\mathbb{P}$ be such that $d(\lambda)= \mathcal{L}(\mathbb{P},\lambda)$. For all $\lambda' \in \R^{m + \ell}$, we have
\begin{align*}
\mathcal{L}(\mathbb{P},\lambda')
= \ & \mathcal{L}(\mathbb{P},\lambda) + \langle \lambda'-\lambda, \mathbb{E}^{\mathbb{P}} [ \zeta ] \rangle
= d(\lambda) + \langle \lambda'-\lambda, z \rangle.
\end{align*}
Minimizing the left-hand side w.r.t.\@ $\mathbb{P}$, we obtain that $d(\lambda') \leq d(\lambda) + \langle \lambda'-\lambda, z \rangle$,
as was to be proved.
\end{proof}

We finish this subsection with a technical result concerning the Lipschitz continuity of the function $V$ defined in \eqref{eq:Abstract_V}.

\begin{lemma} \label{lemm:LipschitzV}
The map $V$ is Lipschitz continuous with modulus $\frac{4M}{\varepsilon}$ on $B_{\varepsilon/2}(0)$.
\end{lemma}

\begin{proof}
Let $z_1$ and $z_2 \in B_{\varepsilon/2}(0)$ with $z_1 \neq z_2$. Take an arbitrary $\mathbb{P}_1 \in \mathcal{P}$ such that $\mathbb{E}^{\mathbb{P}_1} [ \zeta ] + z_1 \in \{ 0_m \} \times \R_{-}^\ell$ (whose existence is ensured by Assumption (A3)).
Consider the half-line $\{ z_1 + \theta (z_2-z_1) \,|\, \theta \geq 0 \}$. It has a unique intersection point $z$ with the boundary of $B_{\varepsilon}(0)$. Let $\theta \geq 0$ be such that $z= z_1 + \theta (z_2-z_1)$. Since $\| z \|_\infty= \varepsilon$ and $\| z_1 \|_\infty \leq \varepsilon/2$, we have $\| z-z_1 \|_\infty \geq \varepsilon /2$ and thus
\begin{equation*}
\theta
~=~ 
\frac{\| z - z_1 \|_\infty}{ \| z_2 - z_1 \|_\infty} 
~\geq~
\frac{\varepsilon}{2 \| z_2 - z_1 \|_\infty}.
\end{equation*}
We also have $\theta > 1$, since the whole segment $\{ z_1 + \theta (z_2-z_1) ~: \theta \in [0,1] \}$ is included in $B_{\varepsilon/2}(0)$.
By (A3), there exists $\mathbb{P} \in \mathcal{P}$ such that $\mathbb{E}^{\mathbb{P}} [ \zeta ] + z \in \{ 0_m \} \times \R_{-}^\ell$.
Let us define $\mathbb{P}_2 = (1-1/\theta) \mathbb{P}_1 + (1/\theta) \mathbb{P}$. Since $1/\theta \leq 1$ and since $\mathcal{P}$ is convex, $\mathbb{P} \in \mathcal{P}$. Moreover, we have $z_2= (1-1/\theta)z_1 + (1/\theta) z$, thus
$\mathbb{E}^{\mathbb{P}_2}[ \zeta ] + z_2 \in \{ 0_m \} \times \R_-^{\ell}$.
It follows that
\begin{align*}
V(z_2) \leq \mathbb{E}^{\mathbb{P}_2}[\xi]
= \mathbb{E}^{\mathbb{P}_1} [ \xi ] + \frac{1}{\theta} \big( \mathbb{E}^{\mathbb{P}}[\xi]- \mathbb{E}^{\mathbb{P}_1}[\xi] \big).
\end{align*}
Using (A2) and $1/\theta \leq 2 \| z_2-z_1 \|_\infty/\varepsilon$, we finally obtain that
\begin{equation*}
V(z_2) \leq V(z_1) + \frac{4M}{\varepsilon} \| z_2- z_1 \|_\infty,
\end{equation*}
which concludes the proof.
\end{proof}

\subsection{Reformulation of the constrained control problem on the canonical space}

In order to prove the duality result in Theorem \ref{thm:duality}, we will reformulate the continuous time and discrete time constrained control problems \eqref{eq:def_V} and \eqref{eq:def_Vh} in the framework of \eqref{eq:primal} in Section \ref{subsec:abstract_duality}.
Concretely, we will reformulate problems  \eqref{eq:def_V} and \eqref{eq:def_Vh} as optimization problems over a space of probability measures on an appropriate canonical space.
Moreover, the space of probability measures enjoys a good closeness property, 
which also plays an essential role to prove the approximation results in Theorem \ref{thm:cvg}.

Recall that $\Om := C([0,T], \R^n)$ denotes the canonical space of all $\R^n$--valued continuous paths on $[0,T]$,
with canonical process  $X = (X_t)_{0 \le t\le T}$, and canonical filtration $\F = (\Fc_t)_{0 \le t\le T}$.	
Denote by $\A$ the collection of all Borel (positive) measures $m$ on $[0,T] \times A$ whose marginal distribution on $[0,T]$ is the Lebesgue measure $ds$, i.e.,
$\A := \{m : m(ds,da) = m(s,d a)ds\}$ for a measurable family of $(m(s, d a))_{s \in [0,T]}$ of Borel probability measures on $A$.
We also consider a subset $\A_0 \subset \A$, which consists of all measures $m(ds ,da)$ such that $m(ds, da) = \delta_{\psi_s}(da) ds$ for some Borel measurable function $\psi \colon [0,T] \to A$.
Let us denote by $\Lambda$ the canonical element on $\A$ and denote 
$$\Lambda_t(\phi) := \int_0^t \int_A \phi(s,a) \Lambda(ds, da)$$ 
for any bounded Borel measurable functions defined on $[0,T] \x A$.
We then introduce a canonical filtration $\F^A = (\Fc^A_t)_{0 \le t \le T}$ on $\A$ by
$$
\Fc^A_t := \sigma \big\{ \Lambda_s(\phi)  ~: s \le t, ~\phi \in C_b([0,T] \x A) \big \}.
$$
Next, we introduce an enlarged canonical space 
$\Omb := \Om \x \A$, which inherits the canonical processes $(X, \Lambda)$,
and is equipped with the enlarged canonical filtration $\Fbb = (\Fcb_t)_{0 \le t \le T}$,
defined by $\Fcb_t := \sigma \big(X_s, \Lambda_s(\phi) ~:s \le t,~\phi \in C_b([0,T] \x U) \big)$.
Denote also $\Fcb := \Fcb_T$, and by $\Pc(\Omb)$ the set of all probability measures on $(\Omb, \Fcb)$,
equipped with the weak convergence topology.

We now introduce a Wasserstein distance on a subspace of $\Pc(\Omb)$. 
Let $\Pc([0,T] \x A)$ denote the space of all probability measures $m$ on $[0,T] \x A$,
and $\Pc_2([0,T] \x A)$ the subspace of measures $m$ such that $\int_0^T \int_A \big( \rho(a, a_0) \big)^2 m(ds, da) < \infty$,
define the Wasserstein distance $d_2$ on $\Pc_2([0,T] \x A)$ by
$$
d_2(m, m') := \Big( \inf_{\pi \in \Pi(m, m')} \int_0^T \int_A\int_0^T \int_A \big( \rho(a, a')^2 + |s - s'|^2 \big) \pi(ds, da, ds', da') \Big)^{1/2},
$$
where $\Pi(m, m')$ denotes the collection of joint distributions on $([0,T] \x A) \x ([0,T] \x A)$ with marginal distribution $m$ and $m'$.
We can then extend the Wasserstein distance $d_2$ on $\A_2 := \{ m \in \A ~: m/T \in \Pc_2([0,T] \x A) \}$ by
$$
d_{\A_2} (m, m') := d_2(m/T, m'/T), ~~\mbox{for all}~m, m' \in \A_2.
$$
Now, let us denote $\Pc_2(\Omb)$ the set of all $\Pb \in \Pc(\Omb)$, such that
\begin{equation} \label{eq:def_Pcb2}
\E^{\Pb} \Big[ \|X \|^2 + \int_0^T \int_A \rho(a_0, a)^2 \Lambda(ds, da)  \Big] < \infty.
\end{equation}
It is easy to check that $\Lambda \in \A_2$, $\Pb$--a.s.\@ for all $\Pb \in \Pc_2(\Omb)$.
We then introduce the Wasserstein distance $\Wc_2$ on $\Pc_2(\Omb)$ by
\begin{equation} \label{eq:def_Wc2}
\Wc_2(\Pb, \Pb') 
~:=~ 
\Big( \inf_{\pi \in \Pi(\Pb, \Pb')}
\int_{\Omb} \int_{\Omb}
\big( \|\om - \om'\|^2 + d_{\A_2}^2(m, m') \big) 
 \pi \big(d (\om, m), d (\om', m') \big) 
\Big)^{1/2}.
\end{equation}

\paragraph{Reformulation of the continuous time control problem}

Following El Karoui, Huu Nguyen and Jeanblanc \cite{EKJ}, we can in fact reformulate the control problem on the canonical space by the martingale problem.
Recall the definition of strong and weak control terms in Definitions \ref{def:weak_ctrl} and \ref{def:strong_ctrl}.
Given a function $\varphi \in C^{2}_{b}(\mathbb{R}^n)$,
let us define an $\Fbb$-adapted continuous function $(C_t(\varphi))_{t \in [0,T]}$ by
\beno
C_t(\varphi) := \varphi(X_t) - \int_{0}^{t}\int_A \Lc \varphi(s, X_{s \wedge \cdot}, a, X_{s}) \Lambda(ds,da),
\eeno
with
\beno
\Lc \varphi(s, \om, a, x)
~:=~ 
\mu(t, \om_{t\wedge \cdot}, a) \cdot D \varphi (x) 
+
\frac{1}{2} \mbox{Tr} \Big[ \sigma \sigma^{T}(t, \om_{t\wedge \cdot}, a) D^{2} \varphi (x) \Big].
\eeno

\begin{definition}
$\mathrm{(i)}$ A probability measure $\Pb$ on $(\Omb, \Fcb)$ is called a relaxed control rule if
\begin{equation} \label{eq:PcR_integrability}
\Pb [ X_0 = x_0] = 1,
~~~
\E^{\Pb} \Big[ \|X \|^2 +  \int_0^T \int_A \rho(a_0, a)^2 \Lambda(ds, da)  \Big] < \infty,
\end{equation}
and the process $(C_t(\varphi))_{t \in [0,T]}$ is a $(\Fbb, \Pb)$-martingale for all functions $\varphi \in C^{2}_{b}(\mathbb{R}^{n})$. \\[1mm]
\noindent $\mathrm{(ii)}$ A  probability measure $\Pb$ is called a weak control rule if there exists a weak control term $\gamma \in \Gamma$ such that $\Pb = \P^{\gamma} \circ \big(X^{\gamma}, \Lambda^{\gamma} \big)^{-1}$ with $\Lambda^{\gamma}(ds ,da) := \delta_{\alpha^{\gamma}_s}(da) ds$. \\[1mm]
\noindent $\mathrm{(iii)}$ A weak control rule $\Pb$ is called a strong control rule if it is induced by a strong control term  $\gamma \in \Gamma_S$,
and it is called a piecewise strong control rule if $\gamma \in \Gamma_{S,0}$.
\end{definition}

Let us denote
\begin{align} 
& \quad \Pcb_R ~\mbox{(resp.}~\Pcb_W, ~\Pcb_S, ~\Pcb_{S,0} \mbox{)} \notag \\
& \quad \qquad :=
\big\{ \mbox{All relaxed (resp.\@ weak, strong, piecewise strong) control rules} \big\},
\label{eq:def_Pcb}
\end{align}
and, for $z \in \R^{m + \ell}$,
\begin{align*}
\Pcb_R(z) 
:= 
\big\{
\Pb \in \Pcb_R ~: \ &
\E^{\Pb}[ \Psi_i(X_{\cdot})] = z_i,
~i=1, \cdots,m, \\
& \E^{\Pb}[ \Psi_{m+j}(X_{\cdot})] \le z_{m+j},
~j=1, \cdots, \ell
 \big\},
\end{align*}
and
$$
\Pcb_W(z) := \Pcb_W \cap \Pcb_R(z).
$$
Consequently, one can redefine $V(z)$ in \eqref{eq:def_V}, and $d(\lambda)$ in \eqref{eq:def_D} by
\be \label{eq:redef_V}
V(z) = \sup_{\Pb \in \Pcb_W(z)} \E^{\Pb} \big[ \Phi( X_{\cdot}) \big],
~~\mbox{and}~
d(\lambda) 
= 
\sup_{\Pb \in \Pcb_W} \E^{\Pb} \big[ \Phi( X_{\cdot}) + \lambda \cdot \Psi(X_{\cdot}) \big].
\ee
Notice that the integrability condition \eqref{eq:def_Pcb2} implies that  $\Pcb_R \subset \Pc_2(\Omb)$.

\begin{lemma} \label{lemm:PS_to_PR}
$\mathrm{(i)}$ Both sets $\Pcb_W$ and $\Pcb_R$ are convex, and 
\begin{equation} \label{eq:Redef_PW}
\Pcb_W ~=~ \big\{ \Pb \in \Pcb_R ~: \P[ \Lambda \in \A_0 ] = 1 \big\},
\end{equation}
and  $\Pcb_{S,0} \subset \Pcb_S \subset \Pcb_W \subset \Pcb_R$. \\[1mm]
\noindent $\mathrm{(ii)}$ Let Assumption \ref{assum:approximation} hold.
Then $\Pcb_{S,0}$ is dense in $\Pcb_R$ under the Wasserstein $\Wc_2$ distance.
\end{lemma}

\begin{proof}
$\mathrm{(i)}$ We first prove \eqref{eq:Redef_PW}.
Given a weak control term $\gamma \in \Gamma$, and let $\Pb  := \P^{\gamma} \circ (X^{\gamma}, \Lambda^{\gamma})^{-1}$ with $\Lambda^{\gamma} := \delta_{\alpha^{\gamma}_s}(da) ds$,
it is straightforward to check that $\Pb \in \Pcb_R$ and satisfies $\Pb[ \Lambda \in \A_0] = \P^{\gamma} [ \Lambda^{\gamma} \in \A_0 ] = 1$.
On the other hand, let $\Pb \in \Pcb_R$ be such that  $\Pb[ \Lambda \in \A_0] =1$, one can construct on a possible enlarged space of $\Omb$, a $\Fbb$--predictable process $\alpha$ and a Brownian motion $W^{\Pb}$ 
such that $\Pb[ \Lambda(da, ds) = \delta_{\alpha_s}(da) ds] =1$ and
$$
X_t = x_0 + \int_0^t \mu(s, X_s, \alpha_s) ds + \int_0^t \sigma(s, X_s, \alpha_s) dW^{\Pb}_s, ~~\Pb \mbox{--a.s.}
$$
It follows that the term $(\Omb, \Fcb, \Fbb, \Pb, X, \alpha, W^{\Pb})$ is a weak control term in Definition \ref{def:Control_h}, and hence $\Pb \in \Pcb_W$.
This proves the equality in \eqref{eq:Redef_PW}.

We next prove that $\Pcb_R$ is convex. Let $\Pb_1, \Pb_2 \in \Pcb_R$, and $\Pb = \theta \Pb_1 + (1-\theta) \Pb_2$ for some $\theta \in [0,1]$.
Then it is clear that $\Pb$ satisfies \eqref{eq:PcR_integrability} as $\Pb_1$ and $\Pb_2$.
Further, as $(C_t(\varphi))_{t \in [0,T]}$ is a $\Fbb$--martingale under both $\Pb_1$ and $\Pb_2$, one has, for all $s \le t$ and $\Fcb_s$--measurable bounded r.v. $\xi$,
\begin{align*}
\E^{\Pb} \big[ \big( C_t(\varphi) - C_s(\varphi) \big) \xi \big] = \ & \theta \E^{\Pb_1} \big[ \big( C_t(\varphi) - C_s(\varphi) \big) \xi \big]  + (1-\theta) \E^{\Pb_2} \big[ \big( C_t(\varphi) - C_s(\varphi) \big) \xi \big] \\
 = \ & 0.
\end{align*}
Thus $(C_t(\varphi))_{t \in [0,T]}$ is also a $\Fbb$--martingale under $\Pb$, and hence $\Pb \in \Pcb_R$.
Then $\Pcb_R$ is convex.
Using \eqref{eq:Redef_PW} and the convexity of $\Pcb_R$, it follows that $\Pcb_W$ is also convex.
Finally, the inclusion relation  $\Pcb_{S,0} \subset \Pcb_S \subset \Pcb_W \subset \Pcb_R$ is trivial by their definitions and \eqref{eq:Redef_PW}. \\[1mm]
\noindent $\mathrm{(ii)}$
Finally, the approximation of a relaxed control rule by weak control rules  is a classical result as illustrated by El Karoui, Huu Nguyen and Jeanblanc \cite{EKJ} under the weak convergence topology.
For the density of $\Pcb_{S, 0}$ in $\Pcb_R$ under $\Wc_2$, 
we can refer e.g.\@ to Theorem 3.1 of Djete, Possamaï and Tan \cite{DPT} for a proof with explicit construction.
\end{proof}

\paragraph{Reformulation of the discrete time control problem}

Notice that a discrete time process on grid $0=t_0 < \cdots < t_N = T$ 
can be considered as a continuous time piecewise constant process on $[0,T]$.
Concretely, let
\begin{equation*}
\gamma = \big( \Om^{\gamma}, \Fc^{\gamma}, \F^{\gamma}, \P^{\gamma}, \alpha^{\gamma}, X^{\gamma}, U^{\gamma}_k, k=1, \cdots, N \big) \in \Gamma_h
\end{equation*}
be a weak discrete time control term,
we (re-)define $(\Xh^{\gamma}, \alpha^{\gamma}, \Lambda^{\gamma})$ as processes on $[0,T]$ by
$$
\Xh^{\gamma}_{s} := (t_{k+1}-s) X^{\gamma}_{t_k} + (s-t_k)  X^{\gamma}_{t_{k+1}},
~~
\alpha^{\gamma}_s := \alpha^{\gamma}_{t_k},
~~\mbox{for all}~s \in [t_k, t_{k+1}),
$$
and $\Lambda^{\gamma}(ds, da) := \delta_{\alpha^{\gamma}_s}(da) ds$.
Denote
\begin{align*}
& \Pcb^h_W
:=
\big\{
\Pb^h := \P^{\gamma} \circ \big(\Xh^{\gamma}, \Lambda^{\gamma} \big)^{-1} 
~:\gamma \in \Gamma_h
\big\}, \\
& \Pcb^h_W(z) 
:= 
\big\{
\Pb^h := \P^{\gamma} \circ \big(\Xh^{\gamma}, \Lambda^{\gamma} \big)^{-1} 
~:\gamma \in \Gamma_h(z)
 \big\},
\end{align*}
so that $V_h$ and $d_h$ in \eqref{eq:def_Vh} and \eqref{eq:dual_h} can be redefined by
\be \label{eq:redef_Vh}
\qquad \quad V_h(z) = \sup_{\Pb^h \in \Pcb^h_W(z)} \E^{\Pb^h} \big[ \Phi( X_{\cdot}) \big]
~~~\mbox{and}~~
d_h(\lambda) 
= 
\sup_{\Pb^h \in \Pcb^h_W} \E^{\Pb^h} \big[ \Phi( X_{\cdot}) + \lambda \cdot \Psi(X_{\cdot}) \big].
\ee

\begin{lemma} \label{lemm:DiscrRelaxForm}
The set $\Pcb^h_W$ is convex.
\end{lemma}

\begin{proof}
Let us consider an enlarged canonical space $\Omh := \Omb \x A^N \x [0,1]^N$ equipped with canonical element $(X, \Lambda, \alpha, U)$ with $\alpha = ( \alpha_0, \cdots, \alpha_{N-1})$ and $U = (U_1, \cdots, U_N)$.
The canonical filtration $\widehat \F = (\widehat \Fc_t)_{0 \le t \le T}$ is defined by
\begin{equation*}
\widehat \Fc_t := \sigma \big(X_s, \Lambda_s(\phi), \alpha_{t_i}, U_i ~: s \le t, \phi \in C_b([0,T] \x A), t_i \le t \big).
\end{equation*}
We define 
$$
\Pch^h_W
:=
\big\{
\P^{\gamma} \circ \big(X^{\gamma}, \Lambda^{\gamma}, \alpha^{\gamma}, U^{\gamma} \big)^{-1} 
~:\gamma \in \Gamma_h
\big\},
~~\mbox{so that}~
\Pcb^h_W = \big\{ \Ph|_{\Omb} ~: \Ph \in \Pch^h_W \big\}.
$$
Then to conclude the proof, it is enough to prove that $\Pch^h_W$ is convex, which implies immediately that $\Pcb^h_W$ is convex.

Let us re-define the process $\alpha$ on $[0,T]$ by $\alpha_s := \alpha_{t_k}$ for all $s \in [t_k, t_{k+1})$ for each $k=0, \cdots, N-1$.
We claim that $\Ph \in \Pch^h_W$ if and only if $\Ph[X_0= x_0] = 1$, and for each $k=1, \cdots ,N$, 
\begin{equation} \label{eq:property_Pch}
\begin{cases}
 X_{t_k} = X_{t_{k-1}} H_h(t_{k-1}, X_{t_{k-1} \wedge \cdot}, \alpha_{k-1}, U_k),\\
X_s = (t_{k+1}-s) X_{t_k} + (s - t_k) X_{t_{k+1}}, ~~s \in [t_k, t_{k+1}], \\
\Lambda(da, ds) = \delta_{\alpha_s}(da) ds, 
~\mbox{and}~U_k \sim \Uc[0,1]~\mbox{is independent of}~\Fch_{t_{k-1}},
\end{cases}
\Ph\mbox{--a.s.}
\end{equation}
Indeed, for every  $\gamma \in \Gamma_h$ and $\Ph = \P^{\gamma} \circ \big(X^{\gamma}, \Lambda^{\gamma}, \alpha^{\gamma}, U^{\gamma} \big)^{-1}$, 
it is straightforward to check that $\Ph$ satisfies $\Ph[X_0= x_0] = 1$ and \eqref{eq:property_Pch}.
Next, let $\Ph$ be a probability measure on $\Omh$ satisfying  $\Ph[X_0= x_0] = 1$ and \eqref{eq:property_Pch},
then it is easy to check that the term $(\Omh, \Fch, \widehat \F, \Ph, (X_{t_k})_{k=0, \cdots, N}, \alpha, U)$ consists of a discrete time weak control term,
and hence $\Ph \in \Pch^h_W$.

We finally prove that $\Pch^h_W$ is convex. Let $\Ph_1, \Ph_2 \in \Pch^h_W$, and $\Ph := \theta \Ph_1 + (1-\theta) \Ph_2$ for some $\theta \in [0,1]$.
Then $\Ph[X_0= x_0] = \theta \Ph_1[X_0= x_0] + (1-\theta) \Ph_2[X_0= x_0] = 1$.
Moreover, \eqref{eq:property_Pch} holds true under $\Ph$, as it does under $\Ph_1$ and $\Ph_2$.
Thus $\Ph \in \Pch^h_W$, which implies that $\Pch^h_W$ is convex, and so is $\Pcb^h_W$.
\end{proof}

\subsection{Proof of the duality results in Theorem \ref{thm:duality}}

Recall that $\Pcb_W$ and $\Pcb_W^h$ are convex by Lemmas \ref{lemm:PS_to_PR} and \ref{lemm:DiscrRelaxForm},
and $(V, d)$  and $(V_h, d_h)$ can be reformulated in \eqref{eq:redef_V} and \eqref{eq:redef_Vh} as in the framework of \eqref{eq:primal} in Section \ref{subsec:abstract_duality}.
Then Theorem \ref{thm:duality} follows directly by
Proposition \ref{prop:duality}, Corollary \ref{coro:charact_dual_optimizer_abstr}, and Lemma \ref{lemma:sufficiency}.

\subsection{Proof of the approximation results in Theorem \ref{thm:cvg}}
\label{subsec:proof_cvg}

To prove the general convergence result in Theorem \ref{thm:cvg}, we will adapt the classical ``weak convergence'' arguments of Kushner and Dupuis \cite{KushnerDupuis} in our constrained context.
In fact, we will work under the Wasserstein $\Wc_2$ distance.
Note that Assumption \ref{assum:approximation} is in force in this subsection.
Let us provide the proof in 4 steps.

\paragraph{Step 1} Recall that $\Pcb_R(z)$ is defined below \eqref{eq:def_Pcb}, we will first prove that
\be \label{eq:def_VR}
V(0) = V_R(0),
~~\mbox{where}~
V_R (z)
:= 
\inf_{\Pb \in \Pcb_R(z)} \mathbb{E}^{\Pb}[\Phi(X_{\cdot})].
\ee

\begin{lemma} \label{lemm:Vb_Lip}
On $B_{\eps/2}(0)$, the functions $V_R(\cdot)$, $V(\cdot)$, and $V_h(\cdot)$, for all $h > 0$, are Lipschitz with Lipschitz constant $\frac{4M}{\eps}$.
\end{lemma}

\begin{proof}
Notice that $\Pcb_R$ is convex, and $\Pcb_W$ is dense in $\Pcb_R$ under $\Wc_2$.
Then it is easy to check that the set $\{ \Pb|_{\Om} ~: \Pb \in \Pcb_R\}$ satisfies all the preliminary conditions in Section \ref{subsec:abstract_duality}, as $\{\Pb|_{\Om} ~:\Pb \in \Pcb_W\}$ does.
Then the Lipschitz property of $V_R$, $V$ and $V_h$ follows by Lemma \ref{lemm:LipschitzV}.
\end{proof}

To prove the equality $V_R(0) = V(0)$ in \eqref{eq:def_VR},
we first notice that  $\Pcb_W(0) \subset \Pcb_R(0)$, and hence $V_R(0) \le V(0)$.

Next, let us fix $\delta > 0$, and $\Pb \in \Pcb_R(0)$ be such that $\E^{\Pb}[ \Phi(X)] \ge V_R(0) - \delta$.
By the density of $\Pcb_W$ in $\Pcb_R$ and the growth condition of $\Phi$ as well as $\Psi_i$ in Assumption \ref{assum:approximation},
there exists some $\Pb' \in \Pcb_W$ such that
$$
\big| \E^{\Pb'} [\Phi(X)] - \E^{\Pb}[ \Phi(X)] \big|  \le \delta,
~~\mbox{and}~
\big| \E^{\Pb'} [\Psi_i(X)] - \E^{\Pb}[ \Psi_i(X)] \big|  \le \delta,
~i=1, \cdots, m+\ell.
$$
This implies that for some $z \in B_{\delta}(0)$, one has $V(z) \le \E^{\Pb'}[\Phi(X)] \le V_R(0) + 2 \delta$.
As $\delta > 0$ is arbitrary, and $V_R$ and $V$ are both Lipschitz on $B_{\eps/2}(0)$,
it follows that $V(0) \le V_R(0)$,
and we hence conclude that $V_R(0) = V(0)$.

\paragraph{Step 2} We next prove that $\liminf_{h \to 0} V_h(0) \ge  V(0) $.
Let $(h_n)_{n \ge 1}$ be a sequence of positive real numbers such that $\lim_{n \to \infty} h_n = 0$ 
and $(\Pb_{n})_{n \ge 1}$ be a sequence such that  $\Pb_n \in \Pcb_W^{h_n}(0)$ for all $n \ge 1$, and
$$
\liminf_{h \to 0} \ V_h(0)
~=~
\lim_{n \to \infty} \E^{\Pb_{n}} \big[ \Phi \big( X_{\cdot} \big) \big].
$$

\begin{lemma} \label{eq:limit_martingale_pb}
The sequence $(\Pb_n)_{n \ge 1}$ is relatively compact under $\Wc_2$.
Moreover, any limit point of $(\Pb_n)_{n \ge 1}$ lies in $\Pcb_R(0)$.
\end{lemma}
\begin{proof}
First, as $A$ is compact, then $\A$ is also compact under the weak convergence topology.
Further, for each $n \ge 1$, let us denote $t^n_k := k h_n$ for $k=0, \cdots, N_n$, then for sequences $(\alpha^n_k, U^n_k)_{0 \le k \le N_n-1}$ (see \eqref{eq:property_Pch})
$$
X_{t^n_{k+1}} = X_{t^n_k} + H_{h_n} (t^n_k,  X_{t^n_k \wedge \cdot}, \alpha^n_k, U^n_k),~~\Pb_n \mbox{--a.s.},
$$
so that
\begin{align*}
\big| X_{t^n_{k+1}} \big|^2 
\le \ &
3 x_0^2 
+ 
3 \, \Big| \sum_{i \le k} \mu(t^n_i, X_{t^n_i \wedge \cdot}, \alpha^n_i) h_n \Big|^2  \\
& \qquad +
3 \, \Big| \sum_{i \le k} \Big(\! H_{h_n}(t^n_i,  X_{t^n_i \wedge \cdot}, \alpha^n_i, U^n_i) - \mu(t^n_i, X_{t^n_i \wedge \cdot}, \alpha^n_i) h_n \!\Big) \Big|^2.
\end{align*}
Let $S^n_k :=  \E^{\Pb_n} \big[ \sup_{0 \le i \le k+1} |X_{t^n_i}|^2 \big]$, using \eqref{eq:growth_condition} and Assumption \ref{assum:approximation}.$\mathrm{(ii)}$,
it follows by direct computation that
$$
S^n_{k+1} ~\le~ \sum_{i \le k} C (S^n_i + 1) h_n,
~~\mbox{for some constant}~C~\mbox{independent of}~n.
$$
Then by the discrete time Gronwall lemma, one has
$$
\sup_{n \ge 1}  \E^{\Pb_n} \Big[ \sup_{0 \le k \le N_n} \big| X_{t^n_k} \big|^2 \Big] < \infty.
$$
Consider the processes $\sum_{i \le k} \mu(t^n_i, X_{t^n_i \wedge \cdot}, \alpha^n_i) h_n $ and $\sum_{i \le k} \sigma \sigma^{\top} (t^n_i, X_{t^n_i \wedge \cdot}, \alpha^n_i) h_n $.
It is easy to deduce that $(\Pb_n|_{\Om})_{n \ge 1}$ is relatively compact under the weak convergence topology, 	by using Theorem 2.3 of \cite{JMM}.
Thus the sequence $(\Pb_n)_{n \ge 1}$ is relatively compact under the weak convergence topology.

We can consider the 3rd moment to compute 
$\E^{\Pb_n} \big[ \big| \sum_{i \le k} \mu(t^n_i, X_{t^n_i \wedge \cdot}, \alpha^n_i) h_n \big|^3 \big]$
and
\begin{align*}
&
\E^{\Pb_n} \Big[ \sup_{k \le N_n}  \Big| \sum_{i \le k} \Big(\! H_{h_n}(t^n_i,  X_{t^n_i \wedge \cdot}, \alpha^n_i, U^n_i) - \mu(t^n_i, X_{t^n_i \wedge \cdot}, \alpha^n_i) h_n \!\Big) \Big| ^3 \Big]\\
& \quad \le
C T^{3/2} \E^{\Pb_n} \Big[ \Big( \frac{h_n}{T} \sum_{k \le N_n} \Big(\! H_{h_n}(t^n_k,  X_{t^n_k \wedge \cdot}, \alpha^n_k, U^n_k) - \mu(t^n_i, X_{t^n_k \wedge \cdot}, \alpha^n_k) h_n \!\Big)^2 \Big)^{3/2} \Big] \\
&\quad \le
C T^{3/2}  \E^{\Pb_n} \Big[ \frac{h_n}{T} \sum_{k \le N_n} \Big|\! H_{h_n}(t^n_k,  X_{t^n_k \wedge \cdot}, \alpha^n_k, U^n_k) - \mu(t^n_i, X_{t^n_k \wedge \cdot}, \alpha^n_k) h_n \!\Big|^3 \Big],
\end{align*}
where the above two inequality follows by BDG inequality and Jensen's inequality.
Then using again Assumption \ref{assum:approximation}.$\mathrm{(ii)}$, one deduces that
$$
\sup_{n \ge 1} \E^{\Pb_n} \big[ \|X \|^3 \big] < \infty,
$$
so that $(\Pb_n)_{n \ge 1}$ is relatively compact w.r.t.\@ the $\Wc_2$ distance (recall that $A$ is compact).

Let $\Pb$ be a limit of $(\Pb_n)_{n \ge 1}$ w.r.t.\@ the $\Wc_2$ distance.	
To prove that $\Pb \in \Pcb_R$, it is enough to prove that, for any $\varphi \in C^2_b(\R^n)$, $0 < s_1 < \cdots s_k < s < t$, $f \in C_b(\R^{(n+p) \x k})$ and $\phi_j \in C_b([0,T] \x A)$, one has
\begin{align*}
& \E^{\Pb} \Big[ f(X_{s_i}, \Lambda_{s_i}(\phi_j); ~1 \le j \le p, 1 \le i \le k)  \big( C_t (\varphi) - C_s(\varphi) \big) \Big] \\
& \qquad = 
\lim_{n \to \infty} \E^{\Pb_n} \Big[ f(X_{s_i}, \Lambda_{s_i}(\phi_j); ~1 \le j \le p, 1 \le i \le k)  \big( C_t (\varphi) - C_s(\varphi) \big) \Big] 
=0.
\end{align*}
To prove the limit property, it is enough to notice that
\begin{eqnarray*}
&& \E^{\Pb_n} \Big[ f(X_{s_i}, \Lambda_{s_i}(\phi_j)~; 1 \le j \le p, 1 \le i \le k)  \big( C_t (\varphi) - C_s(\varphi) \big) \Big]  \\
&=&
\E^{\P^{\gamma_{h_n}}} \Big[  
f \Big( \Xh^{\gamma_{h_n}}_{s_i}, \int_0^{s_i}\phi_j(\alpha^{\gamma_{h_n}}_s)ds); j, i \Big)  \\
&&~~~~~~~~~~~~~
\Big( \varphi \big( \Xh^{\gamma_{h_n}}_t \big) -  \varphi \big( \Xh^{\gamma_{h_n}}_s \big) + \int_s^t \Lc \varphi \big( r, \Xh^{\gamma_{h_n}}_{r\wedge \cdot}, \alpha^{\gamma_{h_n}}_r,  \Xh^{\gamma_{h_n}}_r \big) dr  \Big) 
\Big].
\end{eqnarray*}
Applying Taylor formula on $\varphi \big( \Xh^{\gamma_{h_n}}_t \big) -  \varphi \big( \Xh^{\gamma_{h_n}}_s \big)$, 
then using the dynamic of $X^{\gamma_{h_n}}$ in \eqref{eq:X_gamma_dynamic} and Assumption \ref{assum:approximation}.$\mathrm{(ii)}$,
it follows by direct computation that the limit equals to $0$.
Then $\Pb \in \Pcb_R$.

Finally, $\Pb$ is the limit of $(\Pb_n)_{n \ge 1}$ w.r.t.\@ the $\Wc_2$ distance
and $\Psi_i, ~i=1, \cdots, m+ \ell$ have all quadratic growth, 
thus $\E^{\Pb_n} [ \Psi_i(X_{\cdot}) ] \to \E^{\Pb} [ \Psi_i(X) ]$.
As $\Pb_n \in \Pcb^{h_n}_W(0)$, it follows that $\Pb \in \Pcb_R(0)$.
\end{proof}

Let $(n_k)_{k \ge 1}$ be a subsequence such that $\Pb_{n_k} \to \Pb_* \in \Pcb_R(0)$,
then 
\begin{equation} \label{eq:limitinf}
\liminf_{h \to 0} V_h(0)
~=~
\lim_{k \to \infty} \E^{\P_{n_k}} \big[ \Phi \big( X_{\cdot} \big) \big]
~=~
\E^{\Pb_*} \big[ \Phi \big( X_{\cdot} \big) \big]
~\ge~
V(0).
\end{equation}

\paragraph{Step 3} We prove here $\limsup_{h \to 0} V_h(0) \le V(0)$.

\begin{lemma} \label{lemm:Discret_Approx_Relax}
Let $\Pb \in \Pcb_R(0)$, then for any $\eta_0 > 0$, $\eta > 0$ and $h_0 > 0$, 
there exist $h \le h_0$, $z \in \R^m \x \R^{\ell}_+$ and a weak discrete time control $\gamma_h$
such that $\|z\|_1 \le \eta$, $\gamma_h \in \Gamma_h(z)$ and
$$
\Big| \E^{\Pb} \big[ \Phi(X_{\cdot}) \big] 
- 
\E^{\P^{\gamma_h}} \big[ \Phi \big( \widehat X^{\gamma_h}_{\cdot} \big)  \big] \Big | 
~\le~ 
\eta_0.
$$
\end{lemma}

\begin{proof}
Let  $\Pb \in \Pcb_R(0)$, $\eta_0 > 0$, $\eta > 0$ and $h_0 > 0$.
By Lemma \ref{lemm:PS_to_PR}, there is a piecewise constant strong control $\gamma^{*,1} \in \Gamma_{S,0}$ such that
$\gamma^{*,1} \in \Gamma_S(z_1)$ for some $z_1 \in \R^m \x \R^{\ell}_+$ satisfying $\|z_1\|_1 \le \eta/3$ and 
$ \big|  \E^{\Pb} \big[ \Phi(X_{\cdot}) \big]  - \E^{\P^{\gamma^{*,1}}} \big[ \Phi(X^{\gamma^{*,1}}_{\cdot}) \big] \big| \le \eta_0/3$.

Further, on the probability space $(\Om^{\gamma^{*,1}}, \Fc^{\gamma^{*,1}})$,
one can approximate the piecewise constant control $\alpha^{\gamma^{*,1}}$ by piecewise constant control $\alpha^{*,2}$
such that 
$$
\gamma^{*,1} := (\Om^{\gamma^{*,1}}, \Fc^{\gamma^{*,1}}, \F^{\gamma^{*,1}}, \P^{\gamma^{*,1}}, X^{\gamma^{*,1}}, B^{\gamma^{*,1}}, \alpha^{\gamma^{*,2}}  )\in \Ac(z_2)
$$ 
for some $z_2$ satisfying $\| z_2 - z_1\|_1 \le \eta/3$, 
where each $k \ge 1$, $\alpha^{*,2}_{t_k}$ is a function of $(B^{\gamma^{*,1}}_{t_i} ~:i \le k)$, which is almost surely continuous under the law of $(B^{\gamma^{*,1}}_{t_i} ~:i \le k)$,
and such that
$ \big| \E^{\P^{\gamma^{*,1}}}\big[ \Phi(X^{\gamma^{*,1}}_{\cdot}) \big] -  \E^{\P^{\gamma^{*,1}}}\big[ \Phi(X^{\gamma^{*,2}}_{\cdot}) \big] \big| \le \eta_0/3$.

Next, recall that $\alpha^{*,2}_{t_k}$ are a.s.\@ continuous functions of $(B_{t_i} ~:i \le k)$ for all $k \ge 1$.
One can consider a sequence of discrete time weak controls 
$(\gamma_n)_{n \ge 1}$ such that $\gamma_n \in \Gamma_{h_n}$ for all $n \ge 1$, $h_n \to 0$, and moreover
$$
\alpha^{\gamma_n}_{t_k} = \alpha^{*,2}_{t_k} (B^{\gamma_n}_{t_i}~:i \le k),
$$
where
$$
B^{\gamma_n}_{t_i} := \sum_{j \le i-1} \sigma^{-1} (t_j, \Xh^{\gamma_n}, \alpha^{\gamma_n}_{t_j}) \Big(  H_{h_n} ((t_j, \Xh^{\gamma_n}, \alpha^{\gamma_n}_{t_j}, U^{\gamma_n}_i) - \mu(t_j, \Xh^{\gamma_n}, \alpha^{\gamma_n}_{t_j}) h_n \Big).
$$
Denote $\Lambda^{\gamma_n}(ds, da) := \delta_{\alpha^{\gamma_n}_s}(da) ds$,
then by similar martingale problem arguments  as in Lemma \ref{eq:limit_martingale_pb}, one can check that
$$
\Pb^{\gamma_n} \circ \big( \widehat X^{\gamma_n}_{\cdot}, \widehat B^{\gamma_n}, \Lambda^{\gamma_n} \big)^{-1}
~\to~
\P^* \circ \big(X^*, B^*, \delta_{\alpha^*_s}(da) da \big)^{-1},
$$
for some probability space $(\Om^*, \Fc^*, \P^*)$, equipped with a Brownian motion $B^*$ and $X^*$ satisfies the SDE \eqref{eq:SDE_strong} with Brownian motion $B^*$ and control process $\alpha^*$,
and moreover,
$\alpha^*$ is a piecewise constant control process satisfying
$$
\alpha^*_{t_k} = \alpha^{*,2}_{t_k} (B^*_{t_i}~:i\le k).
$$
Under the Lipschitz condition above \eqref{eq:Lip}, one has the uniqueness of the SDE \eqref{eq:SDE_strong}, then it follows that
$\P^{\gamma_n} \circ \big(\widehat X^{\gamma_n}_{\cdot} \big)^{-1} \to \P^* \circ (X^*)^{-1} = \P_0 \circ  \big( X^{\alpha^{*,2}}_{\cdot} \big)^{-1}$ under $\Wc_2$,
as $n \to \infty$.
Notice that $\Phi$ and $\Psi_i$, $i=1, \cdots, m+ \ell$ are all of quadratic growth,
then there is some discrete time weak control $\gamma_h$ such that
$\gamma_h \in \Gamma_h(z_3)$ for $h \le h_0$ and $\|z_3 - z_2\|_1 \le \eta/3$ and
$$
\Big| \E^{\P_0} \big[ \Phi(X^{\alpha^{*,2}}_{\cdot}) \big] 
- 
\E^{\P^{\gamma_h}} \big[ \Phi \big( \widehat X^{\gamma_h}_{\cdot} \big)  \big] \Big | 
~\le~ 
\eta_0/3.
$$
Finally, combining all the above estimates, it follows that $\gamma_h$ is the required weak discrete time control.
\end{proof}

Let $\Pb^* \in \Pcb_W(0)$ be an $\eps$--optimal solution of the problem $V(0)$, i.e.\@
$$
\E^{\Pb^*}[ \Phi(X_{\cdot}) ] - \eps \le  V(0) = \inf_{\Pb \in \Pcb_W(0)} \E^{\Pb}[ \Phi(X_{\cdot})].
$$
Let $\eta_0 > 0$ be an arbitrary constant,
we use Lemma \ref{lemm:Discret_Approx_Relax} to obtain a sequence $(\gamma_{h_n})_{n \ge 1}$ and a sequence $(z_n)_{ n \ge 1}$
such that  $\gamma_{h_n} \in \Gamma_{h_n}(z_n)$ for all $n \ge 1$, $\|z_n\|_1 \to 0$ and 
$$
\Big| \E^{\Pb^*} \big[ \Phi(X_{\cdot}) \big] 
- 
\E^{\P^{\gamma_{h_n}}} \big[ \Phi \big( \widehat X^{\gamma_{h_n}}_{\cdot} \big)  \big] \Big | 
~\le~ 
\eta_0.
$$
Using the Lipschitz property of $V_h(\cdot)$ in Lemma \ref{lemm:Vb_Lip},
it follows that
\begin{align*}
\limsup_{h \to 0} V_h(0)
~\le~ &
\limsup_{n \to \infty} \ \Big( V_h(z_n) + \frac{4M}{\eps} \|z_n\|_1 \Big) \\
~\le~ &
\limsup_{n \to \infty} \ \E^{\P^{\gamma_{h_n}}} \big[ \Phi \big( \widehat X^{\gamma_{h_n}}_{\cdot} \big) \big] 
~\le~
\E^{\Pb^*} \big[ \Phi(X_{\cdot}) \big] + \eta_0.
\end{align*}
Recall that $\Pb^*$ is an $\eps$--optimal solution of the problem $V(0)$ and $\eta_0 > 0$ is arbitrary,
it follows that
$\limsup_{h \to 0} V_h(0) \le V(0)$.
By \eqref{eq:limitinf}, together with the duality results in Theorem \ref{thm:duality}, one can then conclude that 
$$
D_h(0) = V_h(0) \longrightarrow V(0) = D(0),
~~~\mbox{as}~~
h \longrightarrow 0.
$$

\paragraph{Step 4} We finally prove the convergence rate result in \eqref{eq:cvg_rate}.
Notice that the dual optimizers $\lambda^*$ for problem $D(0) = \sup_{\lambda} d(\lambda)$ and $\lambda^*_h$ for $D_h (0) = \sup_{\lambda_h} d_h(\lambda)$ are both bounded by $\frac{2M}{\eps}$ by Theorem \ref{thm:duality},
then under Assumption \ref{assum:cvg_rate}, one has
$$
|D(0) - D_h (0)| 
~\le~ 
C h^{\rho}.
$$
Finally, \eqref{eq:cvg_rate} follows from  $V_h(0) = D(0)$ and $V(0) = D(0)$ (Theorem \ref{thm:duality}).

\subsection{Proof of Proposition \ref{prop:VS_V}}
\label{subsec:proof_VS_V}

We first notice that $V \le V_S$ as $\Gamma_S \subset \Gamma$.
Further, $V$ is convex on $\R^m \x \R^{\ell}$, and is continuous on $B_{\eps/2}(0)$.
Then
$$
V(z) ~\le~ V^{l.s.c.}_S(z)~\mbox{for}~z \in B_{\eps/2}(0),~~\mbox{and}~~ V(z) ~\le~ V^{conv}_S(z) ~\mbox{for}~z \in \R^m \x \R^{\ell}.
$$
Let $z \in B_{\eps/2}(0)$ so that $V(z) \in \R$.
By similar arguments as in Lemma \ref{lemm:Discret_Approx_Relax}, for any $\Pb \in \Pcb(z)$, 
there is a sequence $(z_n, \Pb_n)_{n \ge 1}$ such that $\Pb_n \to \P$, $z_n \to z$ as $n \to \infty$, and $\Pb_n \in \Pcb_S(z_n)$ for each $n \ge 1$.
It follows that $V(z) \ge \lim_{n \to \infty} V_S(z_n)$ and hence
$$
V(z) ~\ge~ V^{l.s.c.}_S(z), ~\mbox{for}~z \in B_{\eps/2}(0).
$$
Therefore, $V = V^{l.s.c.}_S$ on $B_{\eps/2}(0)$.
By the convexity and Lipschitz continuity of $V = V^{l.s.c.}_S$ on $B_{\eps/2}(0)$, it follows that $V^{l.s.c.}_S$ is the convex envelop of $V_S$ restricted on $B_{\eps/2}(0)$,
which implies that $V(z) = V^{l.s.c.}_S(z) \ge V^{conv}_S(z)$ for $z \in B_{\eps/2}(0)$.
Therefore, one has
$$
V(z) = V^{l.s.c.}_S(z) = V^{conv}_S(z),~\mbox{for all}~z \in B_{\eps/2}(0).
$$

\bibliographystyle{plain}

\end{document}